\documentclass[11pt]{amsart} 
\usepackage{amsmath,amssymb,amsthm,fullpage}
\usepackage{latexsym}
\usepackage{colortbl}
\usepackage{graphicx}
\usepackage[titletoc,toc,title]{appendix}

\theoremstyle{plain}
\newtheorem{theorem}{Theorem}
\newtheorem{lemma}{Lemma}
\newtheorem{proposition}{Proposition}
\newtheorem{problem}{Problem}
\newtheorem{corollary}{Corollary}
\theoremstyle{definition}

\newtheorem{example}{Example}

\newtheorem{definition}{Definition}
\newtheorem{remark}{Remark}

\newcommand{\te}{\theta}

\newcommand{\R}{\mathbb R}

\newcommand{\hide}[1]{} 		

\newcommand{\beq}{\begin{equation}}
\newcommand{\eeq}{\end{equation}}
\newcommand{\beqna}{\begin{eqnarray*}}
\newcommand{\eeqna}{\end{eqnarray*}}
\newcommand{\beqn}{\begin{equation*}}
\newcommand{\eeqn}{\end{equation*}}
\newcommand{\bp}{\begin{proof}}
\newcommand{\ep}{\end{proof}}
\newcommand{\bprop}{\begin{proposition}}
\newcommand{\eprop}{\end{proposition}}
\newcommand{\bt}{\begin{theorem}}
\newcommand{\et}{\end{theorem}}
\newcommand{\bex}{\begin{example}}
\newcommand{\eex}{\end{example}}
\newcommand{\bc}{\begin{corollary}}
\newcommand{\ec}{\end{corollary}}
\newcommand{\bl}{\begin{lemma}}
\newcommand{\el}{\end{lemma}}
\newcommand{\bprob}{\begin{problem}}
\newcommand{\eprob}{\end{problem}}
\newcommand{\br}{\begin{remark}}
\newcommand{\er}{\end{remark}}
\newcommand{\bd}{\begin{definition}}
\newcommand{\ed}{\end{definition}}

\begin{document}

\title[Counterexamples]{
Counterexamples Related to Rotations of Shadows of Convex Bodies}

\author{M. Angeles Alfonseca}
\address{Department of Mathematics, North Dakota State University,
Fargo,  USA} \email{maria.alfonseca@ndsu.edu}

\author{Michelle Cordier}
\address{Department of Mathematics, Kent State University,
Kent, OH 44242, USA} \email{mcordier@math.kent.edu}

\thanks{The first author is supported in part by U.S.~National Science Foundation grant DMS-1100657.}


\maketitle

\begin{abstract} 
We construct  examples of two convex bodies $K,L$ in $\mathbb{R}^n$, such that every projection of $K$ onto a $(n-1)$-dimensional subspace can be rotated to be contained in the corresponding projection of $L$, but $K$ itself cannot be rotated to be contained in $L$. We also find necessary conditions on $K$ and $L$ to ensure that $K$ can be rotated to be contained in $L$ if all the $(n-1)$-dimensional projections have this property.
\end{abstract}

\section{Introduction}

Let $K$ be a convex body in $\mathbb{R}^n$. Given a unit vector $\xi \in S^{n-1}$, we will denote by $K|\xi^\perp$ the orthogonal projection of $K$ on the hyperplane $\xi^\perp=\{x \in \mathbb{R}^n: x \cdot \xi =0\}$. Let $SO(n)$ be the group of rotations in $\mathbb{R}^n$, and  $SO(n-1,\xi^\perp)$ be the group of rotations on the hyperplane $\xi^\perp$. In this paper we study the following problem.

\bprob
 \label{prob2}
{\it  Let $K$, $L$, be  convex bodies in $\mathbb{R}^n$. Suppose that for every $\xi \in S^{n-1}$, the orthogonal projection $K|\xi^\perp$ can be rotated around the origin to fit inside  $L|\xi^\perp$, {\it i.e.} there exists a rotation $\varphi_\xi \in SO(n-1,\xi^\perp)$   such that $\varphi_\xi (K|\xi^\perp) \subseteq L|\xi^\perp$.

 (a) Does it follow that $L$ contains a rotation of $K$, {\it i.e.}, is there a $\psi \in SO(n)$ such that $\psi(K)\subseteq L$?
 
 (b) Does it follow that $vol_n(K) \leq vol_n(L)$?}
\eprob

In  \cite{K1}, D. Klain studied the same questions with translations instead of rotations.   He proved that  the answer to question \ref{prob2}$(a)$ for translations is negative in general, in any dimension. A counterexample is obtained by considering a ball $B$, together with the dilated simplex $(1+\epsilon)T$, where $T$ is the simplex inscribed in $B$. Then, for any $\varepsilon>0$,  the dilated simplex $(1+\varepsilon)T$ is not contained in the ball nor can be translated to fit inside, but if $\varepsilon$ is small enough, all the projections of $(1+\varepsilon)T$ on hyperplanes can be  translated to fit inside the corresponding projections of the ball.  Klain also proved that if both bodies are centrally symmetric, the answer to Problem \ref{prob2}$(a)$ for translations is affirmative.

Regarding question \ref{prob2}$(b)$ for translations, Klain showed that the answer is negative, in general,  but that there  exists a class of bodies such that the answer to \ref{prob2}$(b)$ for translations is affirmative if $L$ belongs to that class.

Problem 1 and the analogous problem for translations are both related to the well-known Shephard's Problem (see \cite{Sh}).

\bigskip

\noindent {\bf Shephard's Problem:}
{\it   Let $K,L$ be origin symmetric convex bodies. If for every $\xi \in S^{n-1}$, we have $vol_{n-1}(K|\xi^\perp) \leq vol_{n-1}(L|\xi^\perp)$, does it follow that $vol_n(K) \leq vol_n(L)$?}

\bigskip

It was proven independently by Petty \cite{P} and Schneider \cite{Sc1} that the answer to Shephard's Problem is negative in general  in dimension $n \geq 3$. In other words, a body $K$ may have greater volume than another body $L$, even if all projections of $K$ have smaller $(n-1)$-dimensional volume than the corresponding projections of $L$. In fact, $K$ may be taken to be a ball, while $L$ is a centrally symmetric double cone (see \cite[Theorem 4.2.4]{Ga}). Petty and Schneider also proved that the answer is affirmative under the additional  assumption that the body $L$ is a {\it projection body} (see \cite[Section 4.1]{Ga}).

Observe  that if $K$ and $L$ are two convex bodies for which the answer to Problem \ref{prob2}$(a)$ is affirmative for either rotations or translations, then Shephard's problem for $K$ and $L$ also has an affirmative answer.

It is natural to consider the analogous question to Problem \ref{prob2}, replacing projections by sections.  Here,  $K\cap \xi^\perp$ denotes the section of $K$ by the hyperplane $\xi^\perp$.

\bprob
 \label{prob3}
{\it  Let $K$, $L$, be convex bodies in $\mathbb{R}^n$. Suppose that for every $\xi \in S^{n-1}$, the  section $K \cap \xi^\perp$ can be rotated around the origin to fit inside $L\cap \xi^\perp$. 

 (a) Does $L$ contain a rotation of $K$? 
 
 (b) Does it follow that $vol_n(K) \leq vol_n(L)$?}
\eprob

In the case of Problem \ref{prob3} for translations, it is known that if $K$ and $L$ are bodies containing the origin in their interior,  such that $K \cap \xi^\perp$ is a translate of $L \cap \xi^\perp$ for every $\xi \in S^{n-1}$, then $K$ is a translate of $L$ (see \cite[Theorem 7.1.1]{Ga}).

We note that Klain's and Shephard's counterexamples will not work for our Problems \ref{prob2} and \ref{prob3},  since in both cases, one of the bodies they consider is a ball, which is invariant under rotations, as are all of its projections and sections. In Section \ref{3d}, we present counterexamples for  Problem \ref{prob2}$(a)$ and Problem \ref{prob3}$(a)$. The first counterexample is in $\mathbb{R}^3$, and consists of a cylinder $C$ and a double cone $K$.  We note that both the cylinder and double cone are centrally symmetric bodies,  and hence, unlike the case of translations proved by Klain, our  Problem \ref{prob2}$(a)$ does not have an affirmative answer for centrally symmetric bodies. The second example, which works in general dimension $n$, is given by  appropriately chosen perturbations of two balls, following ideas of Kuzminykh and Nazarov. However, none of our  counterexamples provide an affirmative answer to Problems \ref{prob2}$(b)$ or \ref{prob3}$(b)$.

In Section \ref{vols}, we obtain a positive answer for Problem \ref{prob2}$(a)$, in $\mathbb{R}^3$, assuming a Hadwiger type additional condition on the bodies $K$ and $L$ (see \cite{Ha}). We also prove that the answer to Problem \ref{prob3}$(b)$ is affirmative (Theorem \ref{MM}).  
However, for the case of projections, the argument only allows us to conclude the relation $vol_n(K^*) \geq vol_n(L^*)$ for the polar bodies, while Problem \ref{prob2}$(b)$ remains open.

In Section \ref{cong}, we study the case in which the projections of $K$ are {\it equal}, up to a rotation, to the corresponding projections of $L$. We prove that if $K$ and $L$ are bodies in $\mathbb{R}^3$ with  countably many diameters, and that their hyperplane projections do not have certain rotational symmetries, then $K=\pm L$. For $n\geq 4$, an $n$-dimensional version of this result has been obtained by the authors and D. Ryabogin in \cite{ACR}, following ideas of Golubyatnikov \cite{Go} and Ryabogin \cite{Rubik}.

{\it Acknowledgements:} We wish to thank Dmitry Ryabogin for many useful discussions, and Fedor Nazarov for the idea of the $n$-dimensional counterexample.

\section{Notation}
\label{nota}

In this section we introduce the notation that will be used throughout the paper.

The unit  sphere in ${\mathbb R}^n$ ($n\ge 2$), is   $S^{n-1}$.  The notation $SO(n)$ for the group of rotations in $\mathbb{R}^n$, and $SO(k)$, $2\le k\le n$ for  their subgroups is standard. If ${\mathcal U}\in SO(n)$ is an orthogonal matrix, we will write ${\mathcal U}^t$ for its transpose.

We will  write $\varphi_{\xi}\in SO(n-1, \xi^\perp)$, meaning that  there exists a  choice of an orthonormal basis in ${\mathbb R}^n$ and a rotation $\Phi\in SO(n)$, with a matrix written in this basis, such that  the action of $\Phi$ on
 $\xi^\perp$ is  the  rotation $\varphi_\xi$ in $\xi^\perp$, and the  action of $\Phi$ in the $\xi$ direction is trivial, {\it i.e.}, $\Phi(\xi)=\xi$. 

We refer to \cite[Chapter 1]{Ga} for the next definitions. A {\em body} in $\R^n$ is a compact set which is equal to the closure of its non-empty interior.  A {\it convex body} is a body $K$ such that for every pair of points in $K$, the segment joining them is contained in $K$.  
For $x \in \mathbb{R}^n$, the {\it support function} of a convex body $K$ is defined as 
$
h_K(x)=\max \{x\cdot y:\, \,y\in K   \}
$
 (see page 16 in \cite{Ga}). The support function uniquely determines a convex body, and $h_{K_1}\leq h_{K_2}$ if and only if $K_1 \subseteq K_2$.
 
 The {\it width function} $\omega_K(x)$ of $K$ in the direction $x\in S^{n-1}$ is defined as
	$\omega_K(x)=h_K(x)+h_K(-x)$. The segment $[z,y]\subset K$ is called the {\it diameter} of the body $K$  if $|z-y|=\max\limits_{\{\theta\in S^{n-1}\}}\omega_K(\theta)$.  Note that a convex body $K$ can have at most one diameter parallel to a given direction, for if $K$ has two parallel diameters $d_1$ and $d_2$, then $K$ contains a parallelogram with sides $d_1$ and $d_2$, and one of the diagonals of this parallelogram is longer than $d_1$. If a diameter of $K$ is parallel to the direction $\zeta\in S^{n-1}$, we will denote it by  $d_K(\zeta)$. We say that a convex body $K\subset {\mathbb R^n}$ has countably many diameters if the width function $\omega_K$  reaches its maximum on a countable subset of $S^{n-1}$.  Also, a body has {\it constant width} if its width function is constant.

A set $E\subset\R^n$ is said to be {\it star-shaped at a point $p$}  if the line segment from $p$ to any point in $E$ is contained in $E$.  
Let $x\in {\mathbb R}^n\setminus \{0\}$, and let $K\subset {\mathbb R}^n$ be a star-shaped set at the origin. The {\it radial function}  $\rho_K$ is defined as 
$
\rho_K(x)=\max \{c:\,cx\in K   \},
$
(here the line through $x$ and the origin is assumed to meet $K$,
\cite[page 18]{Ga}).  We say that a body $K$ is a {\it star body} if it $K$ is star-shaped at the origin and its radial function $\rho_K$ is continuous. 
 The radial function uniquely determines a star body, and $\rho_{K_1} \leq \rho_{K_2}$ if and only if $K_1 \subseteq K_2$.

 For a subset $E$ of $\mathbb{R}^n$, the {\it polar set} of $E$ is defined as $E^*=\{x: x \cdot y \leq 1 \mbox{  for every  }  y\in E\}$ (see  \cite[pages 20-22]{Ga}). When $K$ is a convex body containing the origin, the same is true of $K^*$ (which is called the {\it polar body} of $K$), and we have the following relation between the support function of $K$ and the radial function of $K^*$: For every $u \in S^{n-1}$,
\begin{equation}
 \label{radsup}
      \rho_{K^*}(u)=1/h_K (u).
\end{equation}
For any linear transformation $\phi \in GL(n)$, we have
\begin{equation}
\label{rots}
     h_{\phi K} (u)=h_K (\phi^t u).
\end{equation}
A similar relation 
\begin{equation}
 \label{rotsrho}
\rho_{\phi K} (u)=\rho_K (\phi^{-1} u)
\end{equation}
 holds for the radial function. Combining  (\ref{radsup}),  (\ref{rots}) and (\ref{rotsrho}), it follows that $h_{(\phi K)^*}(u)=h_{\phi^{-t} K^*}(u)$ (see \cite[page 21]{Ga}); this gives us the identity  $(\phi K)^*=\phi^{-t} K^*$ for the polar of a  linear transformation of the body $K$. 

If $S$ is a subspace of $\mathbb{R}^n$, $h_{K|S}(u)=h_K(u)$ and  $\rho_{K\cap S}(u)=\rho_K(u)$, for every $u \in S^{n-1} \cap S$.  Combining these facts with (\ref{radsup}), we obtain the polarity relation between sections and projections, 
\begin{equation}
  \label{secproj}
     \rho_{K^*\cap S}(u)=\rho_{(K|S)^*}(u).
\end{equation}

\section{Counterexamples for Problems  \ref{prob2}$(a)$ and \ref{prob3}$(a)$}
\label{3d}

\subsection{Counterexample in $\mathbb{R}^3$.} Our first counterexample is three-dimensional, and it is provided by two centrally symmetric convex bodies, a cylinder $C$ and a double cone $K$. We show that all {\it sections} of $C$ can be rotated to fit into the corresponding sections of $K$, and that $C$ cannot be rotated to fit inside $K$. Due to the relations  (\ref{rots}), (\ref{rotsrho}), and polarity (\ref{secproj}), this will imply that all projections of  $K^*$ fit into the corresponding projections of  $C^*$ after a rotation, while no rotation of $K^*$ is included in $C^*$. Thus, our two bodies  provide at the same time  counterexamples for \ref{prob2}$(a)$ and  \ref{prob3}$(a)$.

Let $C \subset \mathbb{R}^3$ be the cylinder around the $z$-axis, centered at the origin, with radius $r$ and height $2r$, where \mbox{$\frac{1}{2}<r < \sqrt{2-\sqrt{3}}=0.5176\ldots$} Let $K$ be the double cone obtained by rotating the triangle with vertices $(0,0,\pm1)$  and $(1,0,0)$ around the $z$-axis.  Since $r>1/2$, the cylinder $C$ is not contained in the double cone $K$; in fact, the condition $r>1/2$ is enough to guarantee that no rotation of $C$ is contained in $K$ (the proof of this fact is given in the Appendix, Lemma \ref{no3Drot}). 

Observe that the polar body of $C$ is the double cone obtained by rotating the triangle with vertices $(0,0,\pm 1/r)$ and $(1/r,0,0)$ around the $z$-axis. The polar body of $K$ is a cylinder with radius 1 and height 2. Hence, the dilation of $C^*$ by a factor $r$ is equal to $K$, and similarly the dilation of  $K^*$ by $r$ is equal to $C$. By proving that all sections of $C$ can be rotated to fit into the sections of $K$, we are in fact proving that all projections of $C$ are included in the corresponding projections of $K$ after a rotation. Here we present a sketch of the argument, with the detailed calculations shown in the Appendix.

\begin{figure}[h!]
\centering
\includegraphics[scale=.4,page=1]{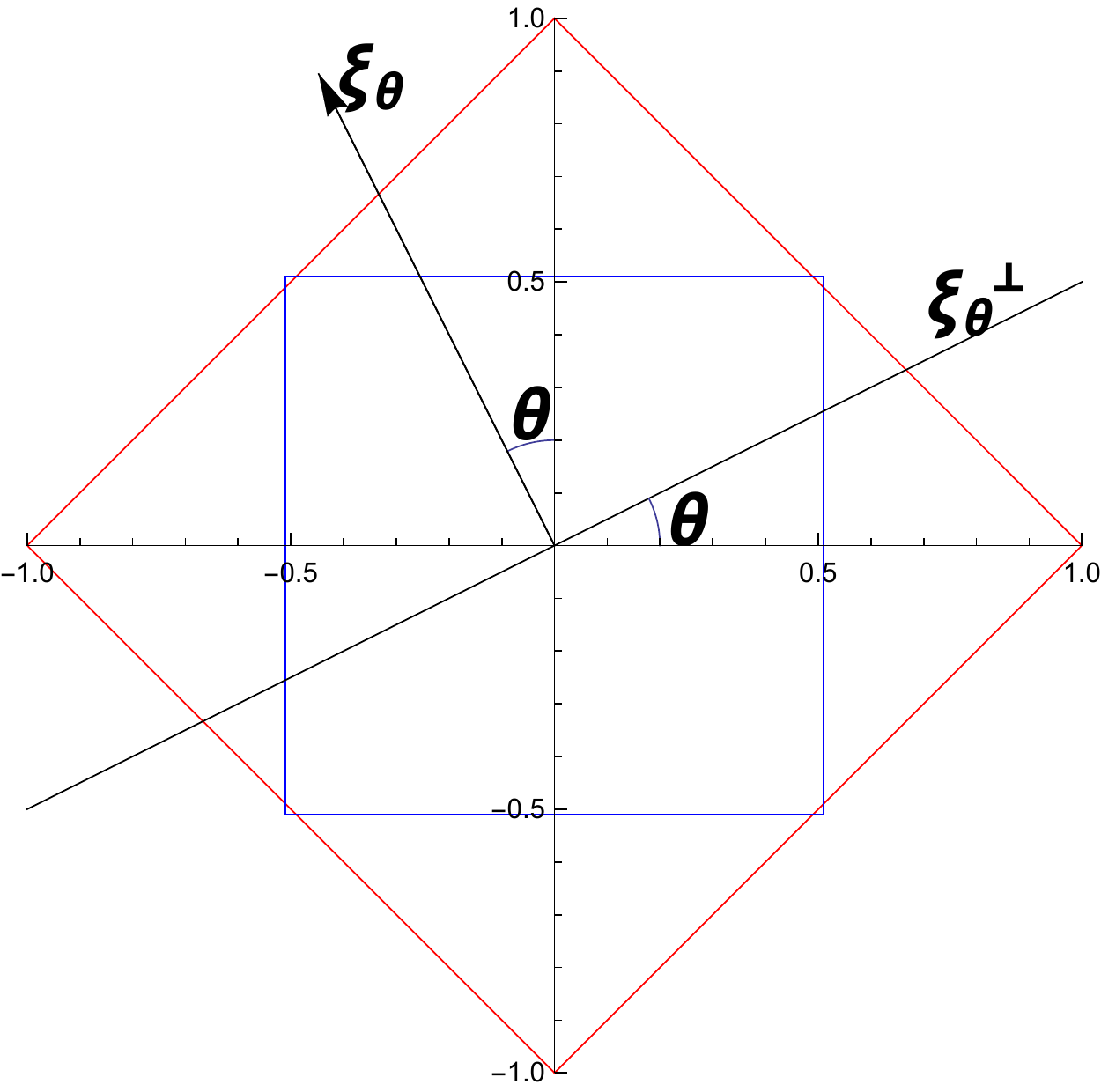}
\caption{Cross-section of the cylinder $C$ and the double cone $K$.}
\label{xsection}
\end{figure}

Since $C$ and $K$ are centrally symmetric bodies of revolution, it is enough to study their sections by planes perpendicular to $\xi_\te=(-\sin(\te),0,\cos(\te))$, where $\te \in [0,\pi/2]$ is the vertical angle from  the axis of revolution (see Figure \ref{xsection}). As seen in the Appendix, the radial function of the section of the double cone $K$ by $\xi_\te^\perp$ is 
\begin{equation}
 \label{radialcone}
    \rho_{K_\te}(u)=\frac{\sec(u)}{\sin(\te)+\sqrt{\tan^2(u)+\cos^2(\te)}}.
\end{equation}
When $\te \in [0,\pi/4]$, the section of the cylinder by $\xi_\theta^\perp$ is an ellipse  with semiaxes of length $r\sec \te$  and $r$. Its radial function is 
\begin{equation}
 \label{radialcyl1}
    \rho_{C_\te}(u)=\frac{r \sec(u)}{\sqrt{\tan^2(u)+\cos^2(\te)}}. 
\end{equation}
On the other hand, when $\te \in (\pi/4,\pi/2]$, the section of the cylinder looks like an ellipse with semiaxes of length $r\sec \te$ (along the $x$-axis) and $r$ (along the $y$-axis), that has been truncated by two vertical lines at $x= \pm r \csc(\te)$. Its radial function is 
\begin{equation}
 \label{radialcyl2}
    \rho_{C_\te}(u)=\left\{ \begin{array}{cc}
   r\sec(u)\csc(\te)    & 0 \leq u \leq u_0, \\
\frac{r \sec(u)}{\sqrt{\tan^2(u)+\cos^2(\te)}} & u_0 \leq u \leq \pi/2,
\end{array}
\right.
\end{equation}
where $u_0=\arctan(\sqrt{\sin^2(\te)-\cos^2(\te)})$. 

\bigskip

Let $\te_0=\arctan\left(\frac{1-r}{r}\right)$. For $\te \in [0,\te_0]$, the section  $C \cap \xi_\te^\perp$ is contained in $K  \cap \xi_\te^\perp$ and there is nothing to prove (see Figure \ref{xsection}).  For $\te \in (\te_0,\pi/4]$,  $C \cap \xi_\te^\perp$ is not a subset of $K  \cap \xi_\te^\perp$. However, a rotation by $\pi/2$  of the section of the cylinder is contained in the section of the double cone. Indeed, from equation (\ref{radialcyl1}) we easily see that the rotation by  $\pi/2$ of the section of the cylinder has radial function 
\begin{equation} 
 \label{radialcyl90}
        \widetilde{\rho}_{C_{\te}}(u)=\frac{r \csc(u)}{\sqrt{\cot^2(u)+\cos^2(\te)}}.
\end{equation}
\begin{figure}[h!]
\centering
\includegraphics[scale=.32,page=1]{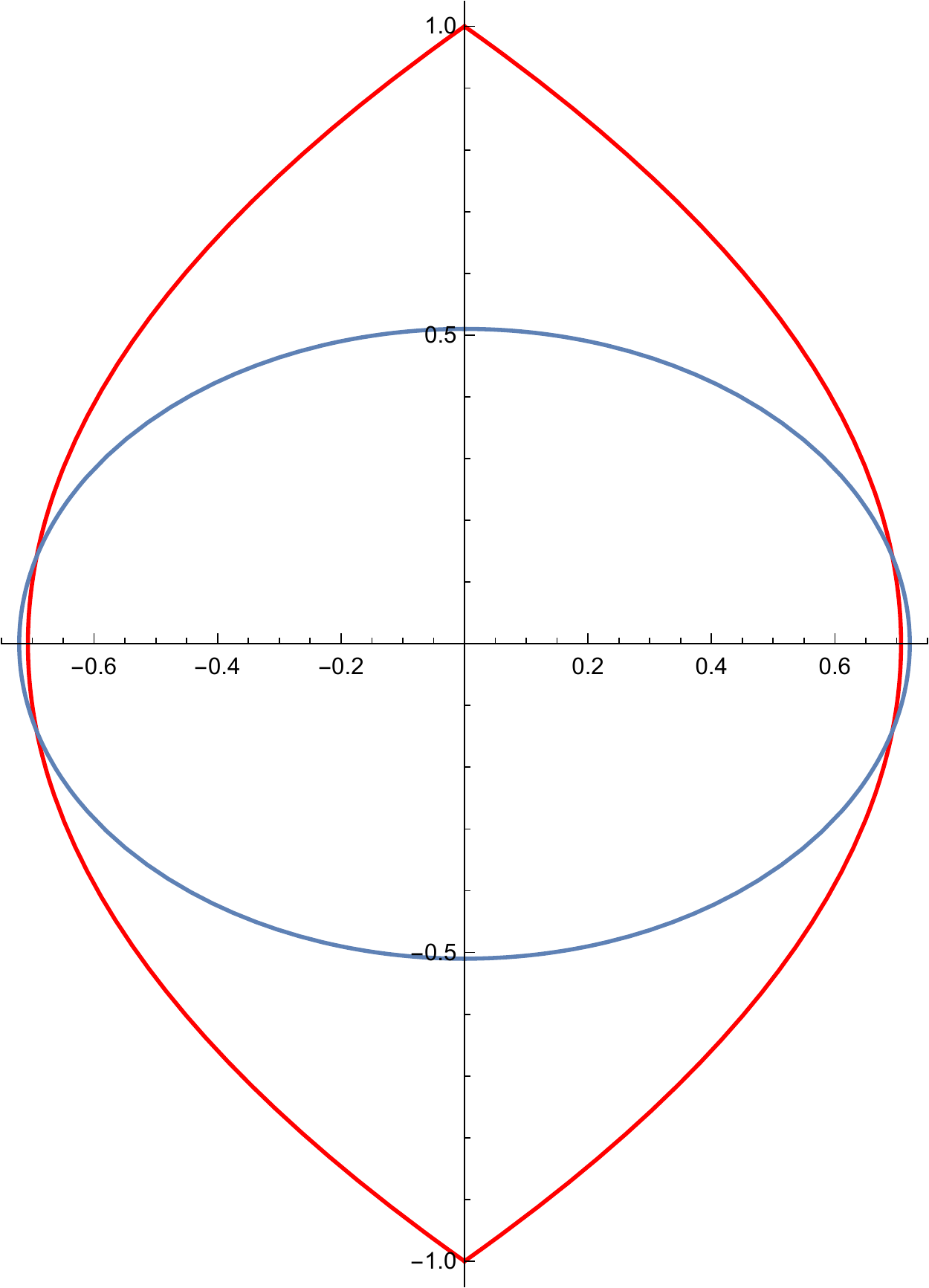}
\includegraphics[scale=.34,page=1]{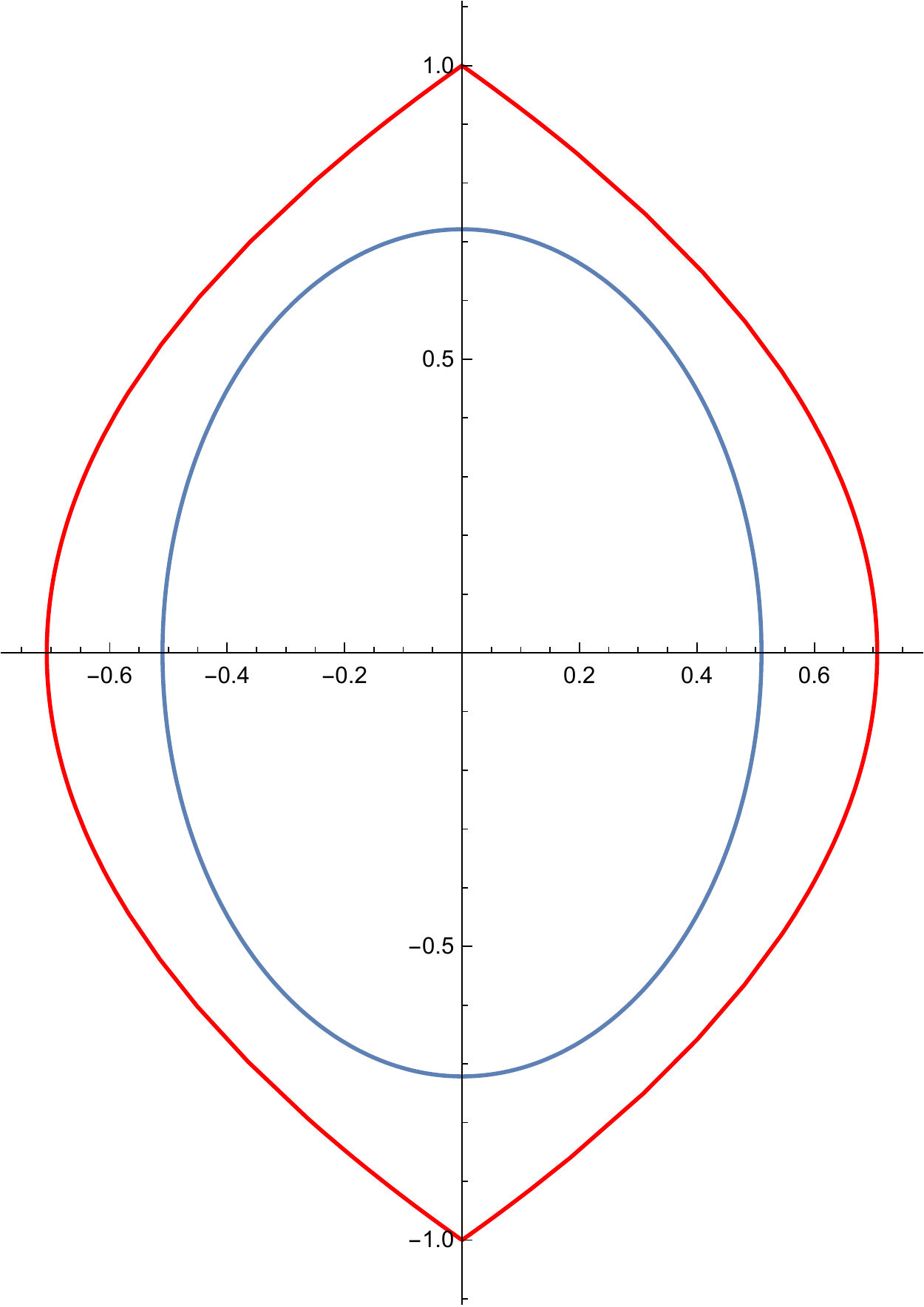}
\caption{Left: For $\te \in (\te_0,\pi/4]$, the section of the cylinder is not a subset of the section of the double cone. Right: The section of the cylinder has been rotated by $\pi/2$. Here $r=0.51$, $\te=\pi/4$. }
\label{rot90}
\end{figure}

We have that  $\widetilde{ \rho}_{C_{\te}}(u) <  \rho_{K_\te}(u)$, for every $u \in [0,\pi/2]$, $\te \in (\te_0,\pi/4]$. The crucial observation is that for fixed $u$, $ \widetilde{\rho}_{C_\te}(u)$ is an increasing 
function of $\te\in [0,\pi/4]$, while $ \rho_{K_\te}(u)$ is decreasing. It is enough, therefore, to show that  $ \widetilde{\rho}_{C_{\pi/4}}(u) <  \rho_{K_{\pi/4}}(u)$  for  $u \in [0,\pi/2]$. Figure \ref{rot90} shows this situation for $r=0.51$. 

\begin{figure}[h!]
\centering
\includegraphics[scale=.45,page=1]{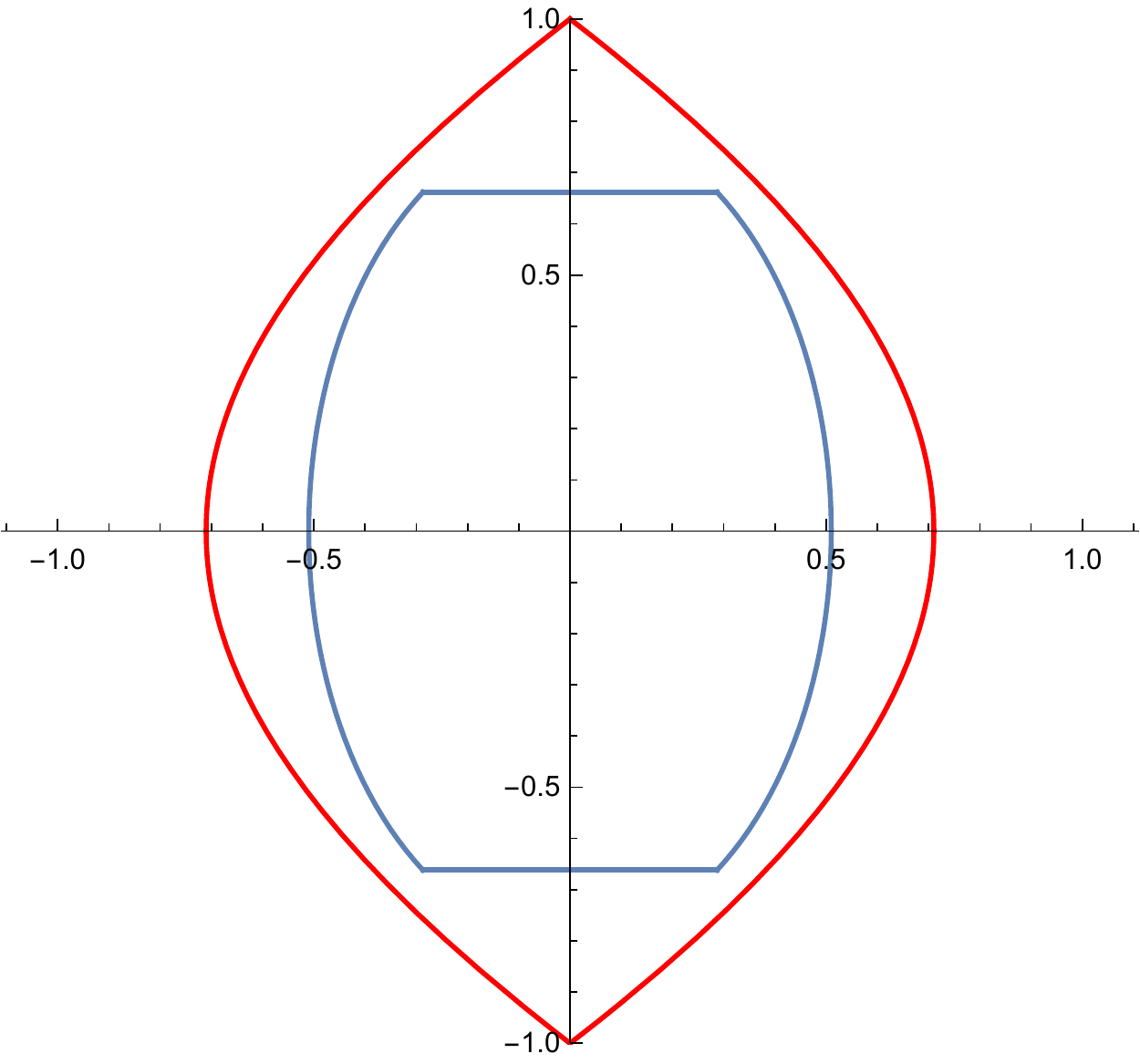}
\includegraphics[scale=.45,page=1]{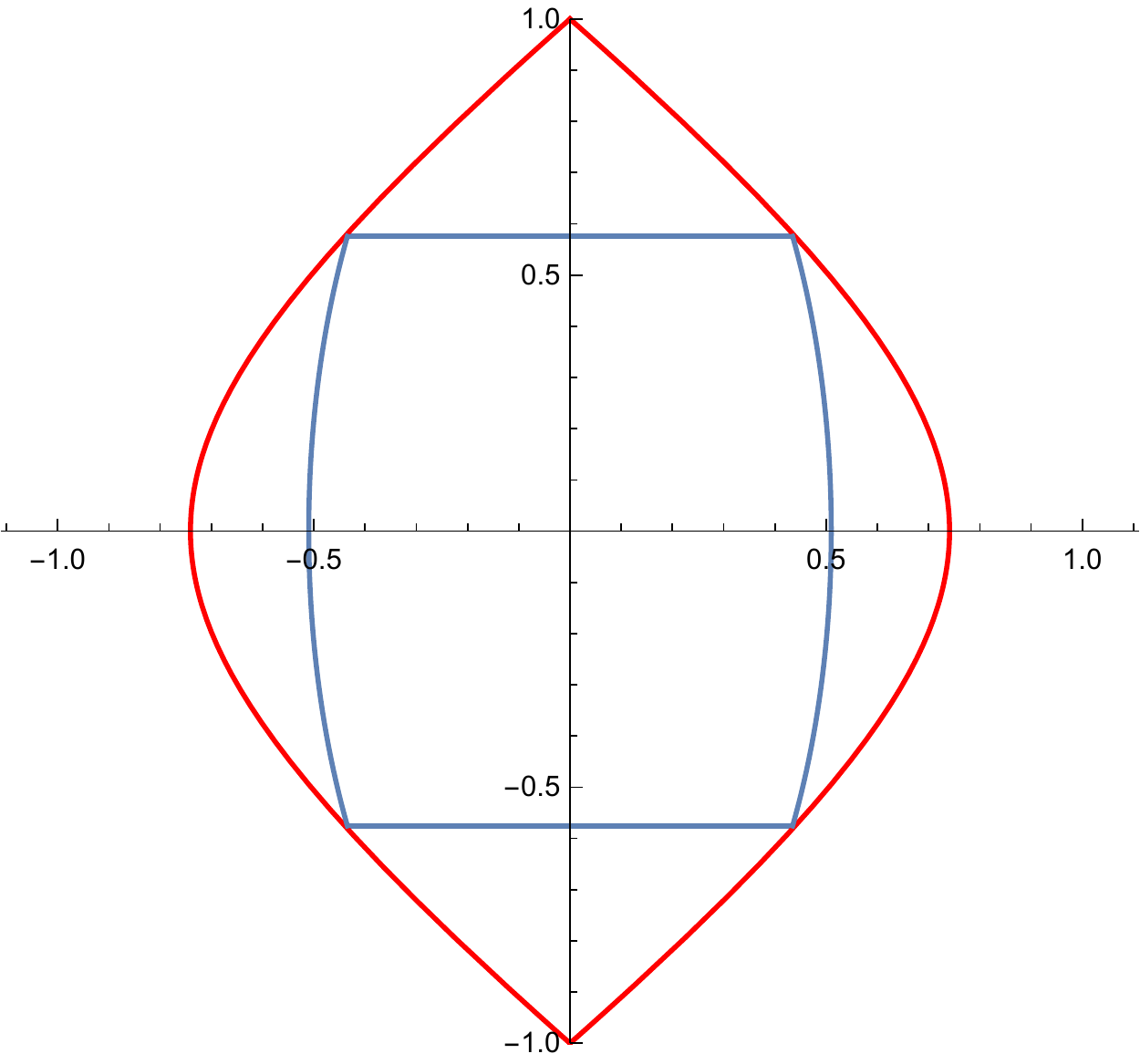}
\caption{Section of the double cone and the $\pi/2$ rotation of the section of the cylinder for $r=0.51$. For the left figure, $\te \in  (\pi/4,\te_1)$; for the right figure $\te=\te_1$. }
\label{3secs}
\end{figure}

\begin{figure}[h!]
\centering
\includegraphics[scale=.45,page=1]{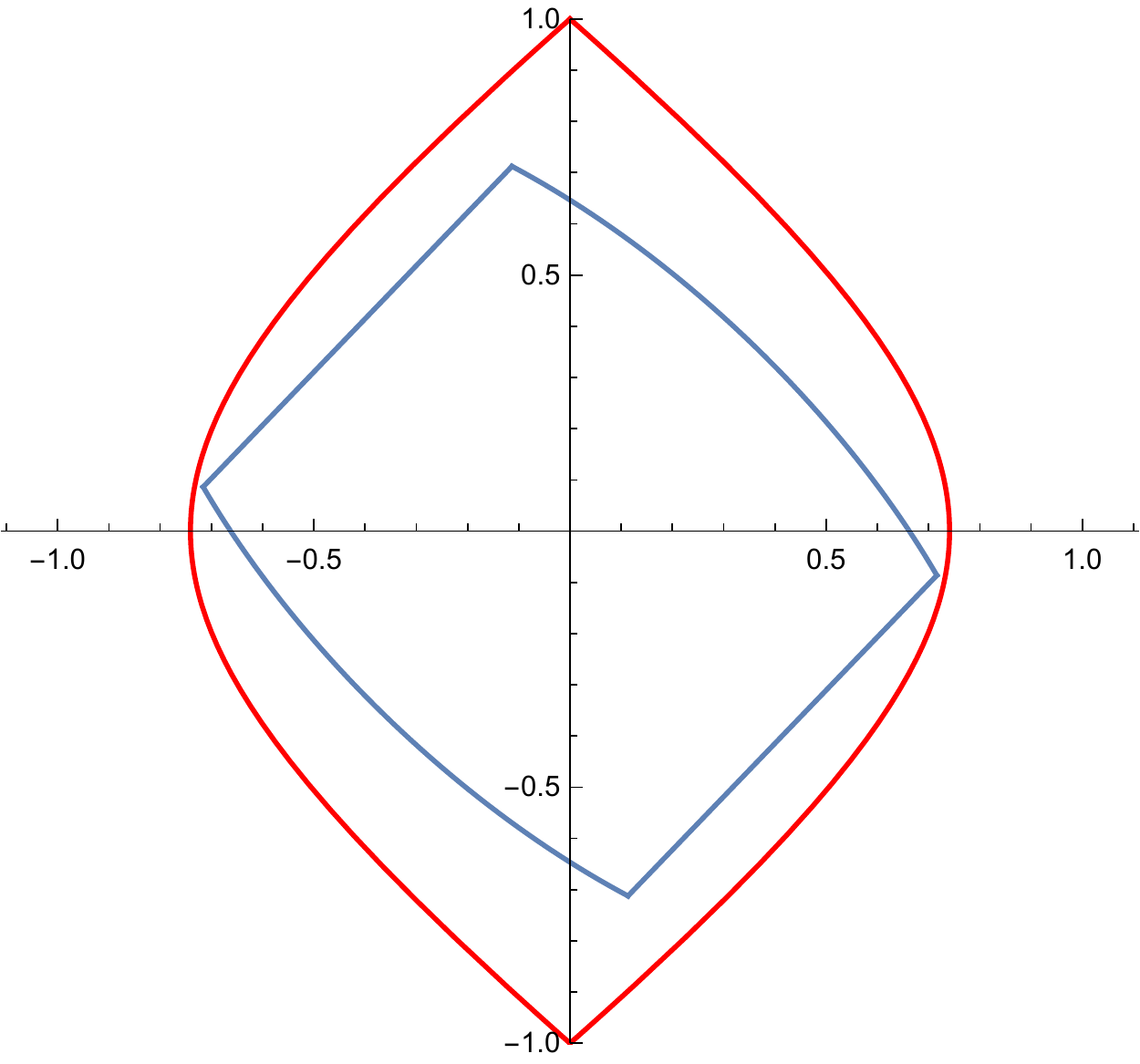}
\includegraphics[scale=.45,page=1]{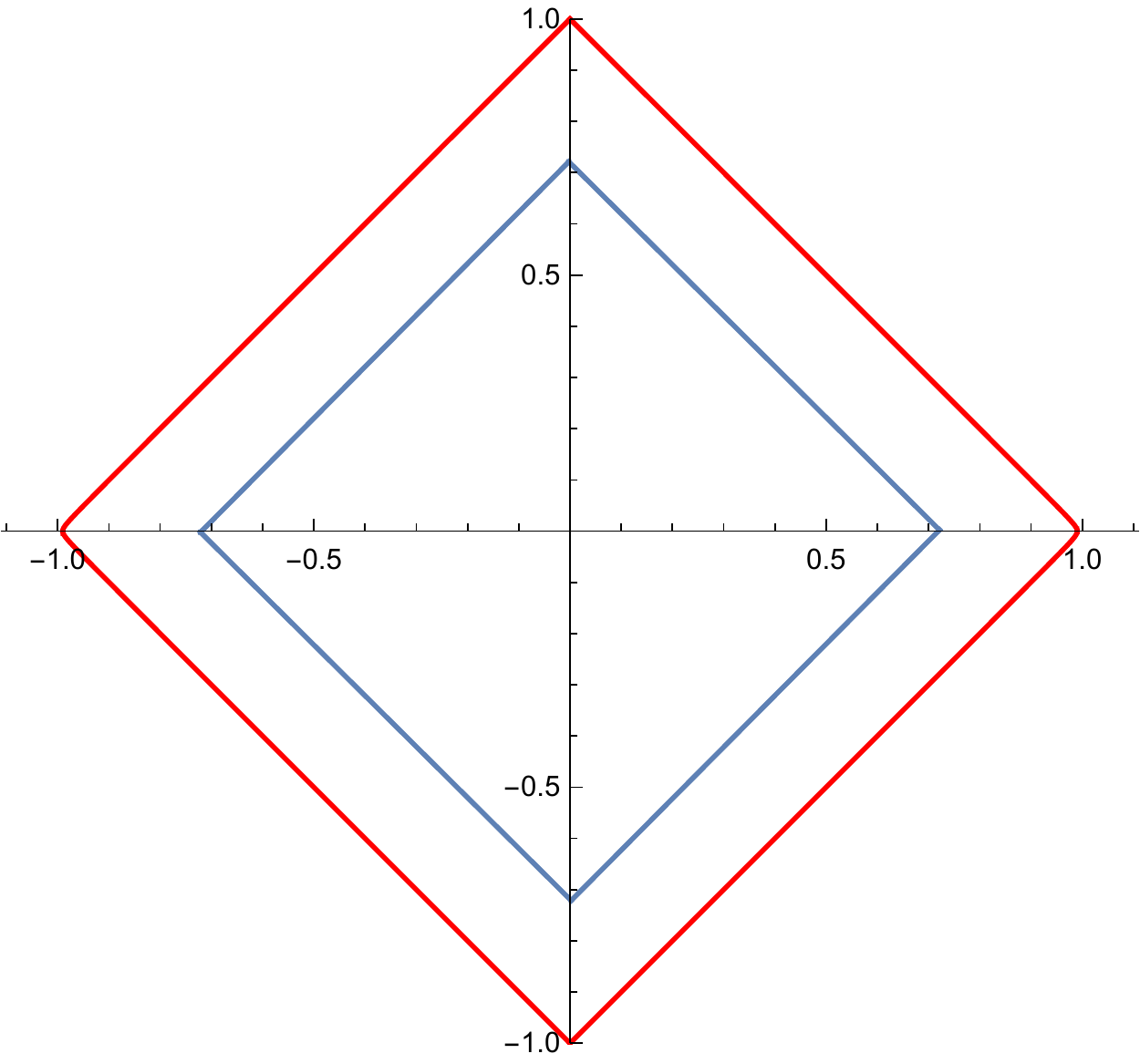}
\caption{In both figures, $r=0.51$. The left figure shows the same sections as Figure \ref{3secs} (right), but the section of the cylinder has been rotated by $\pi/4$. The right figure shows the case where $\te=\pi/2$.}
\label{45d}
\end{figure}

When $\te \in (\pi/4,\pi/2]$, the section $C \cap \xi_\te^\perp$ is never contained in $K  \cap \xi_\te^\perp$, but if $r< \sqrt{2-\sqrt{3}}$,  there exist angles $\te_1, \te_2 \in (\pi/4,\pi/2)$, with $\te_2\leq \te_1$, such that for $\te \in (\pi/4,\te_1]$ a  rotation by $\pi/2$ of the section of the cylinder is contained in the section of the double cone, and for $\te \in[\te_2,\pi/2]$ a rotation by the angle $u_0$ of the section of the cylinder is contained in the section of the double cone. The idea  is that when $\te=\pi/4$, the  $\pi/2$ rotation of the section $C \cap \xi_{\pi/4}^\perp$ is strictly contained  within $K  \cap \xi_{\pi/4}^\perp$, which implies the same, by continuity, for $\te$ in some interval $(\pi/4,\te_1]$;  on the other hand, when  $\te=\pi/2$ and both sections are squares, a rotation by $\pi/4=u_0(\pi/2)$ of $C \cap \xi_{\pi/4}^\perp$ is strictly included in $K  \cap \xi_{\pi/4}^\perp$, and by continuity the same is true on some interval $[\te_2,\pi/2]$. The calculations in the Appendix 
 show that for $r \in (1/2,  \sqrt{2-\sqrt{3}})$, $\te_2 \leq \te_1$ and hence all sections of $C$ can be rotated to fit within the corresponding sections of $K$. Figures  \ref{3secs} and  \ref{45d} illustrate both cases.

\subsection{Counterexample in $\mathbb{R}^n$}

Given the unit sphere in $\mathbb{R}^n$,  we will perturb it by adding bump functions to create two convex bodies $K$ and $L$. We place the bumps on $K$ so that they form a simplex on the surface of $K$, but no such simplex configuration of bumps will appear on the surface on $L$, thus impeding that $K$ may be rotated to fit inside of $L$. On the other hand, every hyperplane section of $K$ will be contained in the corresponding section of $L$ after a rotation. We thus obtain a counterexample for \ref{prob3}$(a)$. By polarity, the bodies $L^*$ and $K^*$ provide a counterexample for Problem \ref{prob2}$(a)$.

\begin{figure}[h!]
\centering
\includegraphics[scale=.7,page=1]{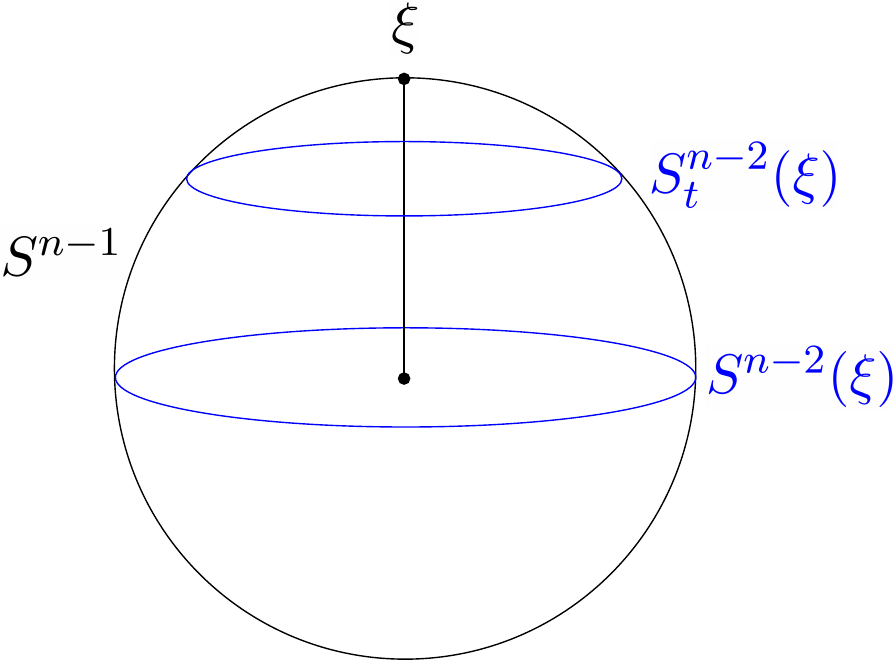}
\caption{}
\label{parallels}
\end{figure}

Given $\xi \in S^{n-1}$, the great $(n-2)$-dimensional subsphere of $S^{n-1}$ that is orthogonal to $\xi$ will be denoted by $S^{n-2}(\xi)=\{ \theta \in S^{n-1}: \theta \cdot \xi =0\}$. For $t\in [-1,1]$, the  subsphere that is parallel to $S^{n-2}(\xi)$ and is located at height $t$ will be denoted by $S^{n-2}_t(\xi)$ (see Figure \ref{parallels}).

The radial function of the unit sphere is the constant function 1. We consider a smooth bump function $\varphi_{\xi,\delta}$ defined on $S^{n-1}$, supported in a small disk $D_\xi$ on the surface of $S^{n-1}$ with center at $\xi \in S^{n-1}$ and with radius $\delta$. The function $\varphi_{\xi,\delta}$ is invariant under rotations that fix the direction $\xi$,  and its maximum height at the point $\xi$ is  1. For $u \in S^{n-1}$, the body whose radial function is $1+\varepsilon \, \varphi_{\xi,\delta}(u)$ is convex, since its curvature will be positive provided that $\varepsilon$ is small enough (see, for example, \cite[page 267]{GRYZ}).  

The first body  $K$ is defined to be the unit sphere with $n$ bumps placed on the surface, so that their centers $\xi_j$, $j=1,\ldots,n$ form a regular spherical simplex.  We assume that the vertex $\xi_1$ is the north pole, and that $v<\frac{4^{-n}}{10^3}$ is the spherical distance between the vertices of the simplex. The radial function of $K$ is 
\[
    1+\sum_{j=1}^n  \frac{\varepsilon}{10^3} \varphi_{\xi_j,\delta}(u), \; {\mbox{ for }}  u\in S^{n-1},
\]
{\it i.e.},  each bump is supported on a disk with center $\xi_j$ and radius $\delta$ (to be chosen later), and has height $\frac{\varepsilon}{10^3}$, where $\varepsilon$ is small enough so that $K$ will be convex.  Given any two vertices of the simplex $\xi_i$ and $\xi_j$, with  $i \neq j$, consider the lune formed by the union of all $(n-2)$-dimensional great spheres passing through any two points $x \in D_{\xi_i}$ and $y \in D_{\xi_j}$.  Let $a$  be the maximum width of the lune. If we choose  $\delta = v^4$, it follows that $a \approx v^3$ (see Figure \ref{lune}).

\begin{figure}[h!]
\centering
\includegraphics[scale=.7,page=1]{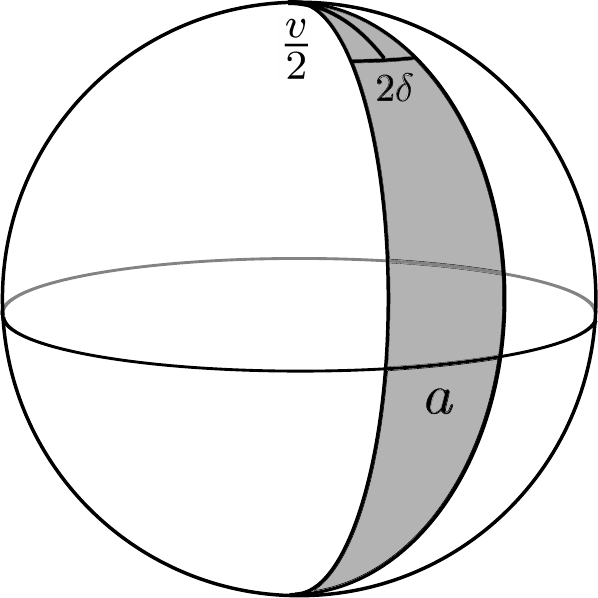}
\caption{}
\label{lune}
\end{figure}

We define $L$ to be the unit sphere with bumps placed on the surface in the following way.  On every point $\xi_j\neq \xi_1$, ({\it i.e.}, not on the north pole), 
we place a bump function of height $\frac{\varepsilon}{10^3}$ and radius  $\delta$. 
Notice that these are the values of the height and radius of the bumps on $K$. 
  Thus, by construction, any section of $K$ that does not pass through the bump at the north pole, is automatically contained in the corresponding section of $L$. 

\begin{figure}[h!]
\centering
\includegraphics[scale=.7,page=1]{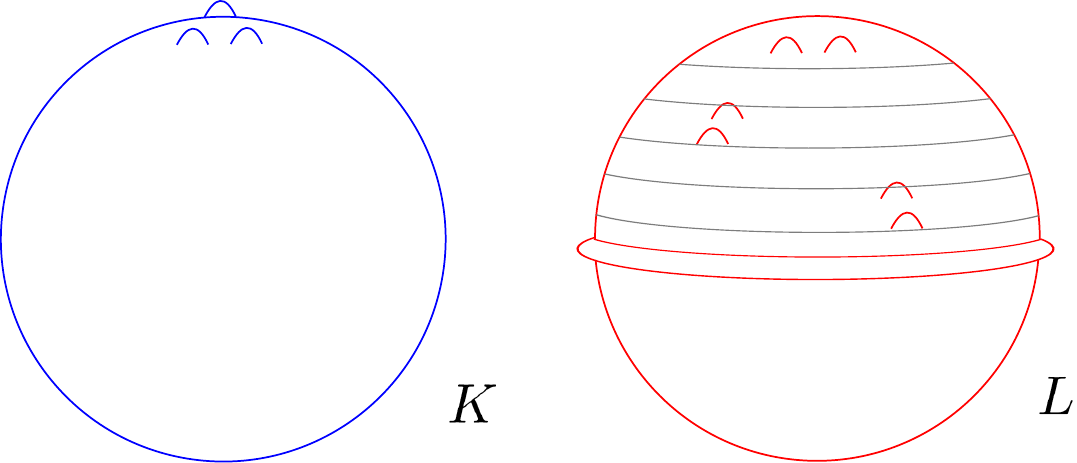}
\caption{Placement of the bumps on $K$ and $L$ for $n=3$.}
\label{bumps}
\end{figure}

We now consider a section of $K$ that passes through the bump at the north pole. We will need to place bumps on $L$ in such a way that the section of $K$ can be rotated to be included in the corresponding section of $L$. For this purpose, we split the top half of the sphere into $2^{n}$ layers. For $k=1,\ldots, 2^n-1$, the $k$-th layer $L_k$ is the spherical ring placed between the parallels $S^{n-2}_{t_{k-1}}(\xi_1)$ and $S^{n-2}_{t_k}(\xi_1)$, where $t_k=\frac{k}{2^{n}}$. The top layer is the spherical cap centered at the north pole, and above the parallel $S^{n-2}_{t_{2^n-1}}(\xi_1)$. Observe that the $(n-1)$ bumps we have already placed are all on the top layer, since $\delta<v$ and $v<\frac{4^{-n}}{10^3}$, while the spherical radius of the top layer is $\frac{1}{2}\arccos(1-\frac{1}{2^n}) \approx \frac{2}{\sqrt{2^{n+1}}}$. 
For every odd $k$, the layer $L_k$  will remain empty of bumps.  For each $2 \leq j \leq n-1$, and for each configuration of $j$ vertices of the simplex in $K$, one of which is the north pole, we place on a layer $L_k$ with $k$ even an identical configuration of vertices ({\it i.e.} a rotation of the original configuration into $L_k$, see Figure \ref{bumps}). On each vertex $x$ we place the bump function $\varepsilon \varphi_{x,\widetilde{\delta}}$, where $\widetilde{\delta}=v^2$. The definition of  $v$ guarantees that the layers are wide enough to contain each configuration of bumps, and also that the larger bumps do not overlap, since $2\widetilde{\delta}<v$.  

Since $\widetilde{\delta}>>a$,  every section of $K$ that intersects $j$ of the bumps 
 will be contained after a rotation in the corresponding section of $L$.  On the other hand, since $\widetilde{\delta}<<v$, no layer can contain $n$ bumps of smaller height $\varepsilon/10^3$ placed in the shape of  the original simplex on $K$.

Finally, we define a function  $\psi_{\xi_1}$ as the function obtained by  sliding  $\frac{\varepsilon}{10^3} \varphi$ around the equator  $S^{n-2}(\xi_1)$  of $L$.  This guarantees that every section of $K$ that passes  through the north pole and no other bump on $K$ can be rotated by the angle $\pi/2$ into the corresponding section of $L$. This concludes the $n$-dimensional counterexample.

\section{Sections, projections, and volumes}
\label{vols}

The counterexamples to Problem \ref{prob2}$(a)$  presented in Section \ref{3d} do not provide a negative answer to Problem \ref{prob2}$(b)$, as in both cases the body with the larger projections also has the larger volume. In fact, Theorem \ref{MM1} below shows that  our assumptions on the projections of $K$ and  $L$ only imply that the volume of the polar body $K^*$ is larger than the volume of $L^*$, but gives us no relation between the volumes of $K$ and $L$. On the other hand, the answer to Problem \ref{prob3}$(b)$ is affirmative: One can obtain the desired relation $vol_n(K) \leq vol_n(L)$ if the {\it sections} of $K$ are assumed to fit into the corresponding sections of $L$ after rotation. This is proved in the following Theorem.

\begin{theorem}\label{MM}
Let $K$ and $L$ be two  star bodies in ${\mathbb R^n}$, $n\ge 2$, such that for every $\xi\in S^{n-1}$, there exists a rotation $\varphi_{\xi}\in SO(n-1,\xi^{\perp})$ such that 
\[
\varphi_{\xi}(K\cap\xi^{\perp})\subseteq L\cap\xi^{\perp}.
\]
Then, 
\begin{equation}\label{mmm}
\textrm{vol}_{n}(K)\le \textrm{vol}_{n}(L).
\end{equation}
\end{theorem}
\begin{proof}
By hypothesis, for every  $\xi\in S^{n-1}$ there exists a rotation $\varphi_{\xi}\in SO(n-1,\xi^{\perp})$ such that 
\[ \rho_{\varphi_{\xi}(K\cap\xi^{\perp})}(\theta)\le \rho_{L\cap\xi^{\perp}}(\theta)\quad\forall\theta\in \xi^{\perp}.
\]
By (\ref{rotsrho}), this is equivalent to 
\[
 \rho_K(\varphi^t_{\xi}(\theta))\le \rho_{L}(\theta)\quad\forall\theta\in \xi^{\perp}.
 \]
Raising to the power $n$, integrating, and using  the rotation invariance of the Lebesgue measure, we obtain
\[
\int\limits_{\xi^{\perp}\cap S^{n-1}}\rho_K^n(\varphi^t_{\xi}(\theta))d\theta=\int\limits_{\xi^{\perp}\cap S^{n-1}}\rho_K^n(\theta)d\theta\le
 \int\limits_{\xi^{\perp}\cap S^{n-1}}\rho_L^n(\theta)d\theta.
\]
Averaging over the unit sphere, we have
$$
\int\limits_{ S^{n-1}}d\xi\int\limits_{\xi^{\perp}\cap S^{n-1}}\rho_K^n(\theta)d\theta\le\int\limits_{ S^{n-1}}d\xi\int\limits_{\xi^{\perp}\cap S^{n-1}}\rho_L^n(\theta)d\theta.
$$
Finally, using Fubini's Theorem and the formula for the volume in terms of the radial function (see \cite[page 16]{K})
\begin{equation}
 \label{vol}
vol_{n}(K)=\frac{1}{n}\int\limits_{ S^{n-1}}\rho^n_K(\theta)d\theta,
\end{equation}
we obtain the result.
\end{proof}

For the next theorem we use the standard notation $int(K)$ for the interior of $K$.  The proof is similar to that of Theorem \ref{MM}.

\begin{theorem}\label{MM1}
Let $K$ and $L$ be two  convex bodies in ${\mathbb R^n}$, $n\ge 2$, such that $0\in int(K)\cap int(L)$, and  for every $\xi\in  S^{n-1}$, there exists a rotation $\varphi_{\xi}\in SO(n-1,\xi^{\perp})$ such that 
\[
\varphi_{\xi}(K|\xi^{\perp})\subseteq L|\xi^{\perp}.
\]
Then, 
\[
\textrm{vol}_{n}(K^*)\ge \textrm{vol}_{n}(L^*).
\]
\end{theorem}
\begin{proof}
By hypothesis,  for every  $\xi\in S^{n-1}$ there exists a rotation $\varphi_{\xi}\in SO(n-1,\xi^{\perp})$ such that 
\[
  h_{\varphi_{\xi}(K|\xi^{\perp})}(\theta)\le h_{L|\xi^{\perp}}(\theta)\quad\forall\theta\in \xi^{\perp}.
\]
By (\ref{radsup}) and (\ref{rots}), this is equivalent to 
\[
 \rho_{K^*}(\varphi^t_{\xi}(\theta))\ge \rho_{L^*}(\theta)\quad\forall\theta\in \xi^{\perp}.
\]
Raising to the power $n$, integrating, and using  the rotation invariance of the Lebesgue measure, we obtain
\[
\int\limits_{\xi^{\perp}\cap S^{n-1}}\rho_{K^*}^n(\varphi^t_{\xi}(\theta))d\theta=\int\limits_{\xi^{\perp}\cap S^{n-1}}\rho_{K^*}^n(\theta)d\theta\ge
 \int\limits_{\xi^{\perp}\cap S^{n-1}}\rho_{L^*}^n(\theta)d\theta.
\]
Averaging over the unit sphere  we have
$$
\int\limits_{ S^{n-1}}d\xi\int\limits_{\xi^{\perp}\cap{ S^{n-1}}}\rho_{K^*}^n(\theta)d\theta\ge\int\limits_{{ S^{n-1}}}d\xi\int\limits_{\xi^{\perp}\cap{ S^{n-1}}}\rho_{L^*}^n(\theta)d\theta.
$$
Finally, using Fubini's Theorem and (\ref{vol}), we obtain the desired result.
\end{proof}

In order to obtain a positive answer to Problem \ref{prob2}$(b)$, we need to impose additional conditions on the bodies $K$ and $L$. We do this in Theorem \ref{countable}, following ideas of Hadwiger \cite{Ha},  by assuming the existence of a diameter $d_K(\xi_0)$ of $K$ in the direction $\xi_0$, such that the hypotheses of Problem \ref{prob2} hold on every plane that contains that diameter.  For $w \in \xi_0^\perp$, we will call  $K|w^\perp$ (resp. $L|w^\perp$) a {\it side projection} of $K$ (resp. of $L$). See Figure \ref{egg}.

\begin{figure}[h!]
\centering
\includegraphics[scale=.7,page=1]{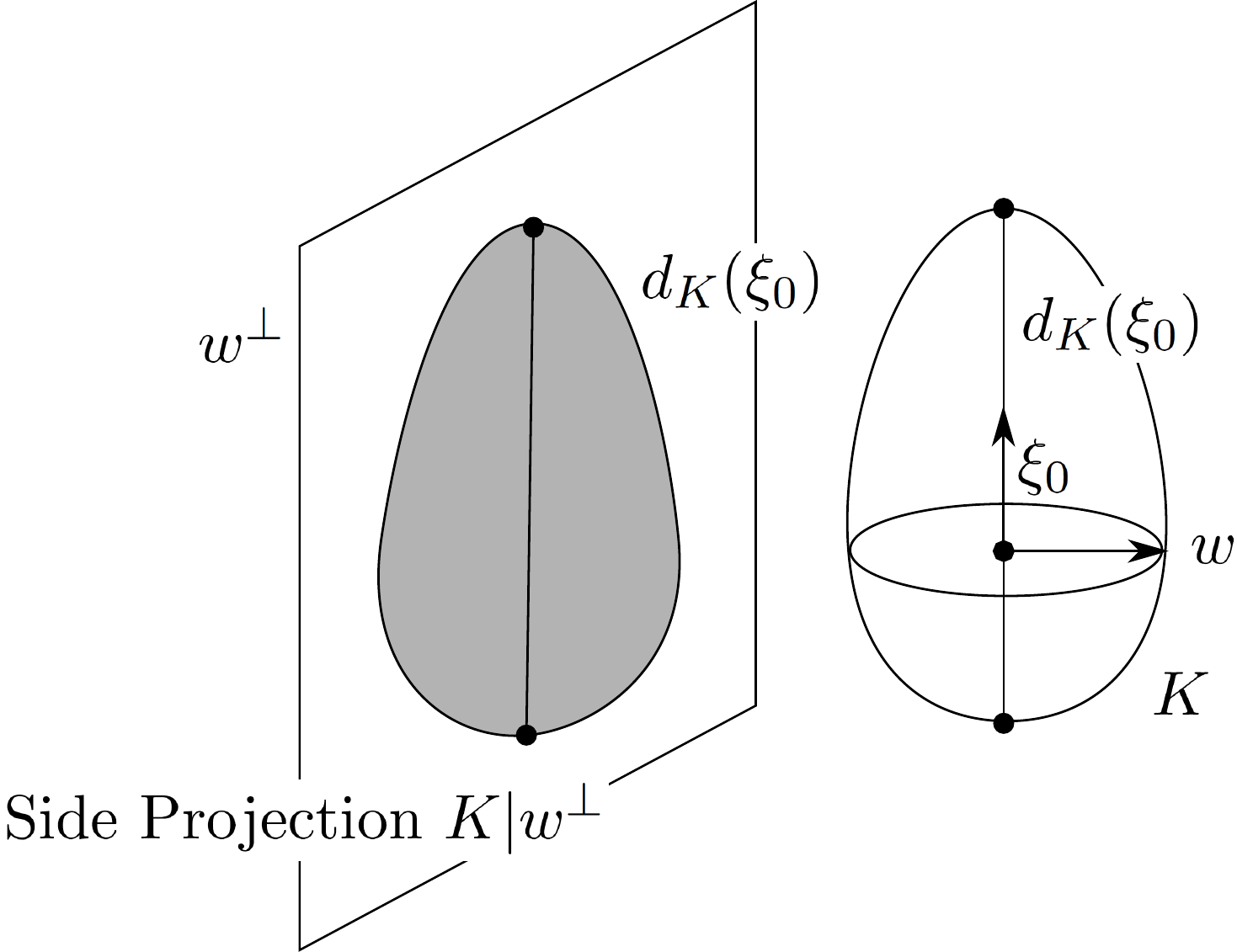}
\caption{}
\label{egg}
\end{figure}

\begin{theorem} \label{countable}
Let $K,L$ be convex bodies in $\mathbb{R}^3$ such that each of $K,L$ have countably many diameters of the same length, {\it i.e.}, $\max\limits_{\{\theta\in S^{n-1}\}}\omega_K(\theta)=\max\limits_{\{\theta\in S^{n-1}\}}\omega_L(\theta)$, and, for each body, the maximum is attained at a countable set of directions.  Assume that there exists a diameter $d_K(\xi_0)$, such that for every $w \in \xi_0^\perp$,  there exists $\varphi_w \in SO(2,w^\perp)$ such that $\varphi_w(K|w^\perp) \subseteq L|w^\perp$, {\it i.e.}, every side projection of $K$ is contained, after a rotation, in the corresponding side projection of $L$.  If either $K$ or $L$ is origin-symmetric then $K \subseteq L$.
\end{theorem}

\begin{proof} 

First we show that $L$ must have a diameter in the  direction $\xi_0$.  If this is not the case, then there exists a plane $H$ that contains $\xi_0$ and none of the directions of the diameters of $L$.  Then $K|H$ contains a diameter of $K$ (namely $d_K(\xi_0)$) and $L|H$ does not contain a diameter of $L$, hence $K|H$ can never be rotated into $L|H$. This contradiction shows that $L$ has a diameter $d_L(\xi_0)$. 

Let $D$  be the countable set of all directions of the diameters of $K$ and  $L$, excluding $\xi_0$. For $w \in \xi_0^\perp$, let $w^\perp$ be a plane containing no direction in $D$ (clearly, $w^\perp$ contains $\xi_0$). Since $\varphi_w(K|w^\perp) \subseteq L|w^\perp$, and $\varphi_w$ is a rotation around the origin that takes the diameter $d_K(\xi_0)$ onto the parallel diameter $d_L(\xi_0)$, it follows that either $\varphi_w$ is the identity, or that both $d_K(\xi_0)$ and $d_L(\xi_0)$ are centered at the origin and $\varphi_w$ is a rotation by $\pi$.  In the first case, we have $K|w^\perp \subseteq L|w^\perp$.  In the second case, we have $-K|w^\perp \subseteq L|w^\perp$. But either $K$ or $L$ is origin-symmetric, and hence we obtain $K|w^\perp \subseteq L|w^\perp$ also in this case. Thus, for every $\theta \in S^2$ such that $\theta \in w^\perp$ and $w^\perp$  does not contain any direction in $D$, we have that $h_K(\theta) \leq h_L(\theta)$.

Let $H_i$ be the plane that contains $\xi_0$ and $\xi_i \in D$, and assume that  $\theta \in S^2 \cap H_i$.  Since there are only countably many such $H_i$'s, we can choose a sequence $\{\theta_j\}$ of points in $S^2$, converging to $\theta$, such that none of the $\theta_j$ are contained in $\displaystyle \cup_{i\geq 1} H_i$.  Hence,  $h_K(\theta_j) \leq h_L(\theta_j)$,  and by  the continuity of the support function, $h_K(\theta) \leq h_L(\theta)$.  We have thus proven that  $K \subseteq L$.
\end{proof}

\br
A more general version of Theorem \ref{countable} can be proven: With the same hypotheses for the diameters of $K,L$, if every projection $K|w^\perp$ can be rotated {\it and translated} to be included in $L|w^\perp$, and either $K$ or $L$ is centrally symmetric, then $K $ is contained in a translate of $L$ (the argument is similar to that in the proof of \cite[Lemma 14]{ACR}).
\er

We finish this section with two related results.

\begin{lemma}
  Let  $K,L$ be two convex bodies in $\mathbb{R}^3$, such that 
$$
\forall\xi\in { S^{2}}, \quad \exists \varphi_{\xi}\in SO(2,\xi^{\perp}) {\mbox{ such that }} \varphi_{\xi}(K|\xi^{\perp})\subseteq L|\xi^{\perp}, 
$$
and
$$
\int_{S^{2}} h_K =\int_{S^{2}} h_L.
$$
Then $K \subseteq \pm L$.  
\end{lemma}

\begin{proof}
  Assume that $\varphi_{\xi}(K|\xi^{\perp})$ is strictly contained in $L|\xi^{\perp}$. By continuity, there is a open set of directions in $S^2$ where the containment is strict. Integrating, we obtain $\int_{S^{2}} h_K <\int_{S^{2}} h_L$, contradicting our hypothesis. Therefore, for every $\xi\in { S^{2}}$ there exists a rotation $\varphi_{\xi}\in SO(2,\xi^{\perp})$ such that $\varphi_{\xi}(K|\xi^{\perp})= L|\xi^{\perp}$. By \cite{Rubik}, we conclude that $K=\pm L$. 
\end{proof}

\begin{lemma}
  Let  $K,L$ be two convex bodies of equal constant width in $\mathbb{R}^3$, such that 
$$
\forall\xi\in { S^{2}}, \quad \exists \varphi_{\xi}\in SO(2,\xi^{\perp}) {\mbox{ such that }} \varphi_{\xi}(K|\xi^{\perp})\subseteq L|\xi^{\perp}. 
$$
Then $vol_3(K)  \leq vol_3(L)$.  
\end{lemma}

\begin{proof}
The assumption on the projections implies that the surface area of $K$ is less than or equal to the surface area of $L$. 
Indeed, Cauchy's surface area formula \cite[page 408]{Ga} says that the surface area  of the body $K$ is equal to 
\[
   S(K)= \frac{1}{vol_{n-1}(B) }\int_{S^{n-1}} vol_{n-1}(K|u^\perp) du,
\] 
where $B$ is the unit Euclidean ball in $\mathbb{R}^n$. Since every projection of $K$ is contained in the corresponding projection of $L$ after a rotation, we have $vol_2(K|u^\perp) \leq vol_2(L|u^\perp)$, and Cauchy's formula gives us $S(K) \leq S(L)$. 

On the other hand, for bodies of constant  width  $w$ in $\mathbb{R}^3$, there is a known formula relating volume, surface area and width (see \cite[page 66]{CG}):
\[
      2 vol_3(K)=w S(K)-\frac{2\pi}{3} w^3.
\]
Using this relation, we conclude that $vol_3(K)  \leq vol_3(L)$.  
\end{proof}

\section{Directly Congruent Projections}
 \label{cong}
 
 In the previous sections, we have studied the problem in which the projections of $K$ can be rotated to be contained into the corresponding projections of $L$. In this section, we will assume that the projections of $K$ are directly congruent to the corresponding projections of $L$ ({\it i.e.} they coincide up to a rotation and a translation). The main argument of the proof follows ideas from Golubyatnikov \cite{Go} and  Ryabogin \cite{Rubik}. Some of the arguments we use can be found in full detail in the paper  \cite{ACR}.

\begin{theorem}\label{tpr}  Let $K$ and $L$ be two convex bodies in ${\mathbb R}^3$ having countably many diameters.  Assume that there exists a diameter $d_K(\xi_0)$, such that the side projections $K|w^\perp$, $L|w^\perp$ onto all subspaces $w^\perp$ containing $\xi_0$ are directly congruent.  Assume also that these projections are not centrally symmetric.  Then $K=\pm L +b$ for some $b \in {\mathbb R}^3$.
\end{theorem}
 


\bp
For the same reasons as in the proof of Theorem \ref{countable}, $L$ must have a diameter parallel to the direction $\xi_0$.  Denote by  $\{ \xi_0, \xi_1,\ldots \}$ the countable set of directions parallel to the diameters of $K$ and the diameters of $L$. Let $H_i$ be the plane that contains the directions $\xi_0$ and $\xi_i$, and consider the set
$$\Lambda =\left\{ w \in S^1(\xi_0): w \notin \bigcup_ i H_i  \right\}.$$
Because the set $\{H_i\}$ is countable, $\Lambda$ is everywhere dense in $S^1(\xi_0)$.
(In fact, the hypothesis that $K$ and $L$ have countably many diameters can be replaced by the weaker hypothesis that $\Lambda$ is everywhere dense.)

Next, we  translate $K$ and $L$ so that their diameters $d_K(\xi_0)$ and $d_L(\xi_0)$ coincide  and are centered at the origin.  We name the translated bodies $\tilde{K}$ and $\tilde{L}$. It is easy to show that  the side projections $\tilde{K}|w^\perp$ and $\tilde{L}|w^\perp$ coincide up to a rotation (see \cite[Lemma 14]{ACR}). 

Denote by $H'_w$ the plane that contains $w$ and $\xi_0$.  Since for all $w \in \Lambda$, the only diameter of $\tilde{K}|H'_w$ is $d_{\tilde{K}}(\xi_0)$, the direct rigid motion given by the statement of the theorem must fix this diameter,  and hence it is either the identity or a rotation about the origin by $\pi$.  In the first case, $\tilde{K}|H'_w=\tilde{L}|H'_w$ and in the second, $\tilde{K}|H'_w=-\tilde{L}|H'_w$.  Define
$$\Xi = \{ w \in S^1(\xi_0): \tilde{K}|H'_w=\tilde{L}|H'_w \}$$ 
and 
$$\Psi = \{ w \in S^1(\xi_0): \tilde{K}|H'_w=-\tilde{L}|H'_w \}.$$

By a standard argument, it is easy to prove that $\Xi$ and $\Psi$ are closed (see Lemma 3 in \cite{ACR}).  From the definitions of $\Xi$ and $\Psi$ we have that  $\Lambda \subseteq \Xi \cup \Psi \subseteq S^1(\xi_0)$, and since $\Xi \cup \Psi$ is closed and $\Lambda$ is everywhere dense we have $\Xi \cup \Psi = S^1(\xi_0)$.

Next we claim that $\Xi \cap \Psi = \emptyset$.  Indeed, assume this is not the case and let $w \in \Xi \cap \Psi$.  Then we have 
$$\tilde{L}|H'_w=\tilde{K}|H'_w=-\tilde{L}|H'_w,$$
which implies that $\tilde{L}|H'_w$ (and hence $L|H'_w$) is centrally symmetric.  This contradicts our assumption, and thus $\Xi \cap \Psi = \emptyset$.

Therefore, either $\Xi=S^1(\xi_0)$ or $\Psi=S^1(\xi_0)$.  If $\Xi=S^1(\xi_0)$, then for every $\theta \in S^2$ there exists a plane  $H'_w$ for some $w \in S^1(\xi_0)$, such that $\theta \in H'_w$.  Hence, 
$$h_{\tilde{K}}(\theta)=h_{\tilde{K}|H'_w}(\theta)=h_{\tilde{L}|H'_w}(\theta)=h_{\tilde{L}}(\theta),$$
where the second equality follows from the definition of $\Xi$. It follows that $\tilde{K}=\tilde{L}$.

If $\Psi=S^1(\xi_0)$, again for every $\theta \in S^2$, we have that $\theta \in H'_w$ for some $w \in S^1(\xi_0)$. Note  that we also have $-\theta \in H'_w$.  Hence,
$$h_{\tilde{K}}(\theta)=h_{\tilde{K}|H'_w}(\theta)=h_{-\tilde{L}|H'_w}(\theta)=h_{(-\tilde{L})|H'_w}(\theta)=h_{-\tilde{L}}(\theta),$$
where the second equality follows from the definition of $\Psi$. We conclude that $\tilde{K}=-\tilde{L}$, and the Theorem is proved. 
\ep

\begin{appendix}
\section{The sections of the cylinder and the double cone in $\mathbb{R}^3$}
\label{cal-cylcone}

Here we provide the detailed calculations for the example in Section \ref{3d} for the convenience of the reader.

\bigskip

\noindent{\bf Determining the radial function of the boundary curves of the sections of $K$ and $C$. }

\bigskip

The upper half of the double cone has equation $z=1-\sqrt{x^2+y^2}$, and the plane $\xi_\te^\perp$ has equation $z=\tan(\te)x$.  The curve of intersection in parametric equations is given by  
\[
 r_{K,\te}(t)=\langle (1-z)\cos(t),(1-z)\sin(t),z \rangle
\]
 where $z=\tan(\te) (1-z)\cos(t)$ (from the equation of the plane). Solving for $z$ in this last equation, we obtain $z=\frac{\tan(\te)\cos(t)}{1+\tan(\te)\cos(t)}$, and therefore
\[
      r_{K,\te}(t)=\left\langle \frac{\cos(t)}{1+\tan(\te)\cos(t)},\frac{\sin(t)}{1+\tan(\te)\cos(t)},\frac{\tan(\te)\cos(t)}{1+\tan(\te)\cos(t)} \right\rangle.
\]
This curve is still expressed as a subset of $\mathbb{R}^3$, so now we will write it as a two dimensional curve on the plane $\xi_\te^\perp$. The vectors $\langle 1,0,0 \rangle$ and $\langle 0,1,0 \rangle$ project onto $\vec{e}_{1,\te}=\left\langle \frac{1}{\sqrt{1+\tan^2(\te)}},0,\frac{\tan(\te)}{\sqrt{1+\tan^2(\te)}} \right\rangle=\langle \cos(\te),0,\sin(\te) \rangle$ and $\vec{e}_{2,\te}=\langle 0,1,0 \rangle$ on the plane $z=\tan(\te)x$. Therefore, for $t \in [0,\pi/2]$, the parametric curve written in this basis becomes
\[
       \widetilde{r}_{K,\te}(t)= \left( \frac{\cos(t)\sec(\te)}{1+\tan(\te)\cos(t)} \right)\vec{e}_{1,\te}+ \left(\frac{\sin(t)}{1+\tan(\te)\cos(t)} \right)\vec{e}_{2,\te}.
\]
Finally, it will be more convenient to express it in polar coordinates. Setting 
\newline $ \widetilde{r}_{K,\te}(t)=\rho_{K_\te}(u) \cos(u)\vec{e}_{1,\te}+ \rho_{K_\te}(u)\sin(u)\vec{e}_{2,\te}$ and solving, we obtain that the radial function of the section $K \cap \xi_\te^\perp$ is 
\[
    \rho_{K_\te}(u)=\frac{\sec(u)}{\sin(\te)+\sqrt{\tan^2(u)+\cos^2(\te)}},
\]
for $u \in [0,\pi/2]$. The function is extended evenly to $[-\pi/2,0]$. It can easily be checked that $\rho_{K_\te}'(u) \geq 0$ when $\te \in [0,\pi/4]$, and thus $\rho_{K_\te}(u)$ is an increasing function of $u$ on $[0,\pi/2]$, with minimum value $\rho_{K_\te}(0)=\frac{1}{\sin \te +\cos \te}$, and maximum value $\rho_{K_\te}(\pi/2)=1$. Also, for fixed $u \in [0,\pi/2]$, $\rho_{K_\te}$ is a decreasing function of $\te \in  [0,\pi/4]$. In contrast, when $\te \in (\pi/4,\pi/2]$, $\rho_{K_\te}(u)$  has a local maximum at $u=0$ and a local (and absolute) minimum at $u_0=\arctan(\sqrt{ \sin^2(t) - \cos^2(t)})$, with value $\rho_{K_\te}(u_0)=1/\sqrt{2}$. Its absolute maximum  is  $\rho_{K_\te}(\pi/2)=1$.

\bigskip

Similarly, we calculate the radial function of $C \cap \xi_\te^\perp$. The intersection of the cylinder with the plane $z=\tan(\te)x$, for $\te \in [0,\pi/4]$, is an ellipse with parametrization 
\[
      r_{C,\te}(t)=\left\langle r\cos(t), r\sin(t), r\cos(t) \tan(\te) \right\rangle.
\]
In terms of the basis $\{\vec{e}_{1,\te},\vec{e}_{2,\te}\}$, the parametrization is given by 
\[
       \widetilde{r}_{C,\te}(t)=  r\cos(t)\sec(\te)\vec{e}_{1,\te}+ r\sin(t) \vec{e}_{2,\te},
\]
and the radial function is 
\[
    \rho_{C_\te}(u)=\frac{r \sec(u)}{\sqrt{\tan^2(u)+\cos^2(\te)}},
\]
and evenly extended on $[-\pi/2,0]$. The section is an ellipse with semiaxes of length $r\sec \te$ (for $u=0$) and $r$ (for $u=\pi/2$), and the radial function is strictly decreasing on $u \in [0,\pi/2]$. It is also useful to note that for fixed $u$, $\rho_{C_\te}$ is an increasing function of $\te \in [0,\pi/4]$.

When  $\te \in [\pi/4,\pi/2]$, the plane cuts the top and bottom of the cylinder, and we obtain the following radial function:
\[
    \rho_{C_\te}(u)=\left\{ \begin{array}{cc}
   r\sec(u)\csc(\te)    & 0 \leq u \leq u_0, \\
\frac{r \sec(u)}{\sqrt{\tan^2(u)+\cos^2(\te)}} & u_0 \leq u \leq \pi/2,
\end{array}
\right.
\]
where $u_0=\arctan(\sqrt{\sin^2(\te)-\cos^2(\te)})$. The section looks like an ellipse with semiaxes of length $r\sec \te$ (along the $x$-axis) and $r$ (along the $y$-axis), that has been truncated by two vertical lines at $x= \pm r \csc(\te)$. The function $\rho_{C_\te}(u)$ has a local minimum at $u=0$,  is strictly increasing on $(0,u_0)$, reaches a local (and absolute) maximum at $u=u_0$ with $\rho_{C_\te}(u_0)=\sqrt{2}r$, and is decreasing on $(u_0,\pi/2)$. The absolute minimum  is $\rho_{C_\te}(\pi/2)=r$. Observe that the absolute maximum of $\rho_{C_\te}$ occurs at the same point as the absolute minimum of $\rho_{K_\te}$, and that $\sqrt{2}r=\rho_{C_\te}(u_0)>\rho_{K_\te}(u_0)=1/\sqrt{2}$, since $r>1/2$, thus reflecting the fact that for $\te >\pi/4$, the section of the cylinder is not contained in the section of the cone. Figure \ref{plots} shows the graphs of $\rho_{K_\te}(u)$ and $\rho_{C_\te}(u)$ with $u\in[0,\pi/2]$, for $r=0.51$. On the left, $\te=\pi/4$; on the right,  $\pi/4<\te<\pi/2$.

\begin{figure}[h!]
\centering
\includegraphics[scale=.5,page=1]{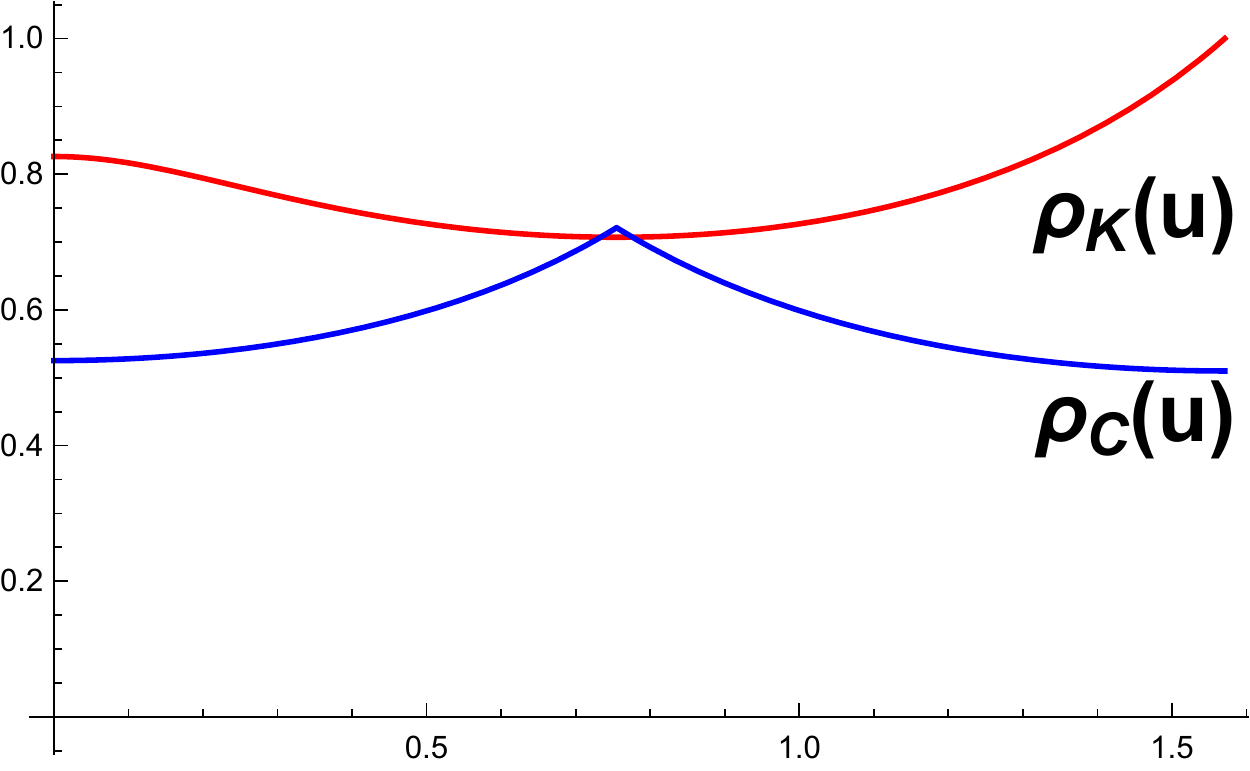}
\includegraphics[scale=.5,page=1]{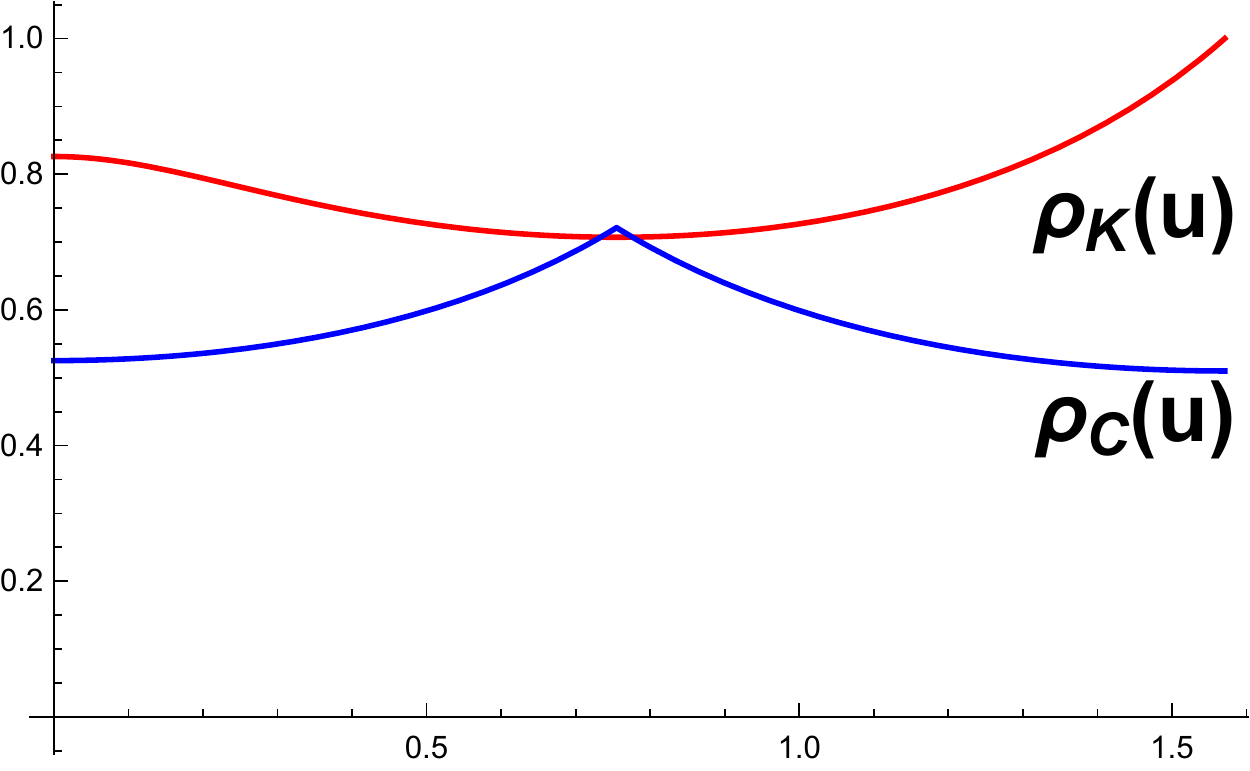}
\caption{Left: $\te=\pi/4$; Right: $\pi/4<\te<\pi/2$.}
\label{plots}
\end{figure}

\bigskip

Now we are ready to compare the sections of the cylinder and the cone on each plane $\xi_\te^\perp$. As noted in Section \ref{3d}, if $\te_0=\arctan\left(\frac{1-r}{r}\right)$, for $\te \in [0,\te_0]$, the section of the cylinder is contained in the section of the cone and there is nothing to prove. For $\te \in (\te_0,\pi/4]$, the section of the cylinder is not contained in the section of the cone, but a rotation by $\pi/2$ of the section of the cylinder is contained in the section of the cone. Since for fixed $u$, $ \widetilde{\rho}_{C_\te}(u)$ is increasing as a function of $\te\in (\te_0,\pi/4]$, while $ \rho_{K_\te}(u)$ is decreasing, it is enough to show that  for  $u \in [0,\pi/2]$,
 $ \widetilde{\rho}_{C_{\pi/4}}(u)<  \rho_{K_{\pi/4}}(u)$. Here, $\widetilde{\rho}_{C_{\pi/4}}(u)$ is the radial function of the rotation by $\pi/2$ of the section of the cone, as defined  in equation (\ref{radialcyl90}). We want to show that
 \begin{equation}
  \label{rot1}
     \frac{r^2 \csc^2(u)}{\frac{1}{2}+\cot^2(u)}   <  \frac{\sec^2(u)}{\left(\frac{1}{\sqrt{2}}+\sqrt{\tan^2(u)+\frac{1}{2}}\right)^2}.
 \end{equation}
This can be rearranged  as 
\[
       r^2 \left(1+\tan^2(u)+\sqrt{2}\sqrt{\tan^2(u)+\frac{1}{2}}\right) < \tan^2(u)\left(\frac{1}{2}+\cot^2(u)\right),
\]
or
\[
         \sqrt{2}\,  r^2  \sqrt{\tan^2(u)+\frac{1}{2}} < \tan^2(u)\left( \frac{1}{2}-r^2 \right)+(1-r^2).
\]
Squaring both sides, we obtain
\[
     0< \frac{1}{4}(1-2r^2)^2 \tan^4 u +  (1-3r^2) \tan^2 u +(1-2r^2), 
\]
a quadratic equation on $\tan^2 u$ whose discriminant is $(1-3r^2)^2-(1-2r^2)^3=r^4(8r^2-3)$. But this expression is negative for $r\in (\frac{1}{2},\sqrt{2-\sqrt{3}})$, and thus (\ref{rot1}) holds.

\noindent{\\ \bf Calculation of the angles $\te_1,\te_2$.}\\

As noted in Section \ref{3d}, when $\te=\pi/4$, the rotation by $\pi/2$ of $C \cap \xi_\te^\perp$ is strictly contained in the section of the double cone, and by continuity the same is true for $\te \in (\pi/4,\te_1)$ for some angle $\te_1$. Similarly, for $\te=\pi/2$ the rotation of the section of the cylinder by $u_0=\pi/4$ is strictly contained in the section of the double cone, and thus the same must hold for $\te \in (\te_2,\pi/2)$. Here we compute $\te_1$ and $\te_2$, and prove that $\te_2 <\te_1$, allowing us to always rotate the section of the cylinder to fit into the section of the cone.

\begin{figure}[h!]
\centering
\includegraphics[scale=.5,page=1]{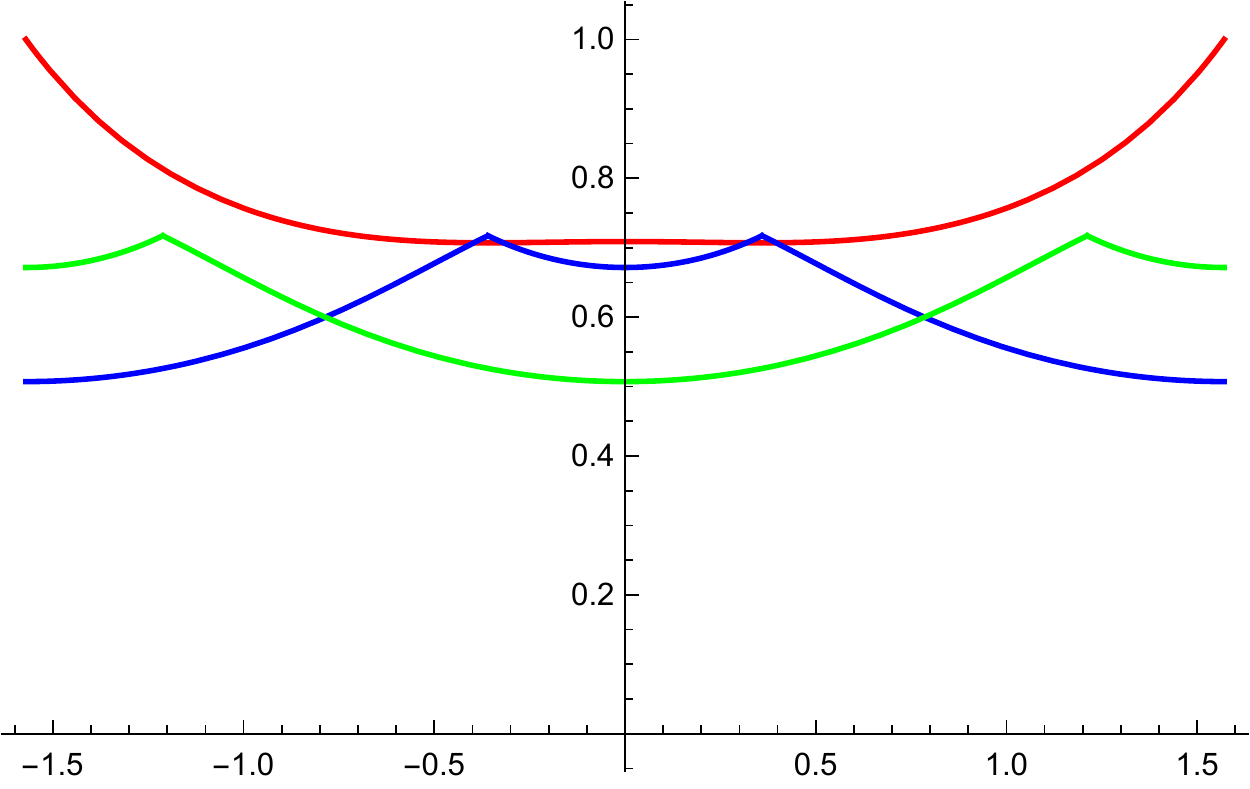}
\includegraphics[scale=.5,page=1]{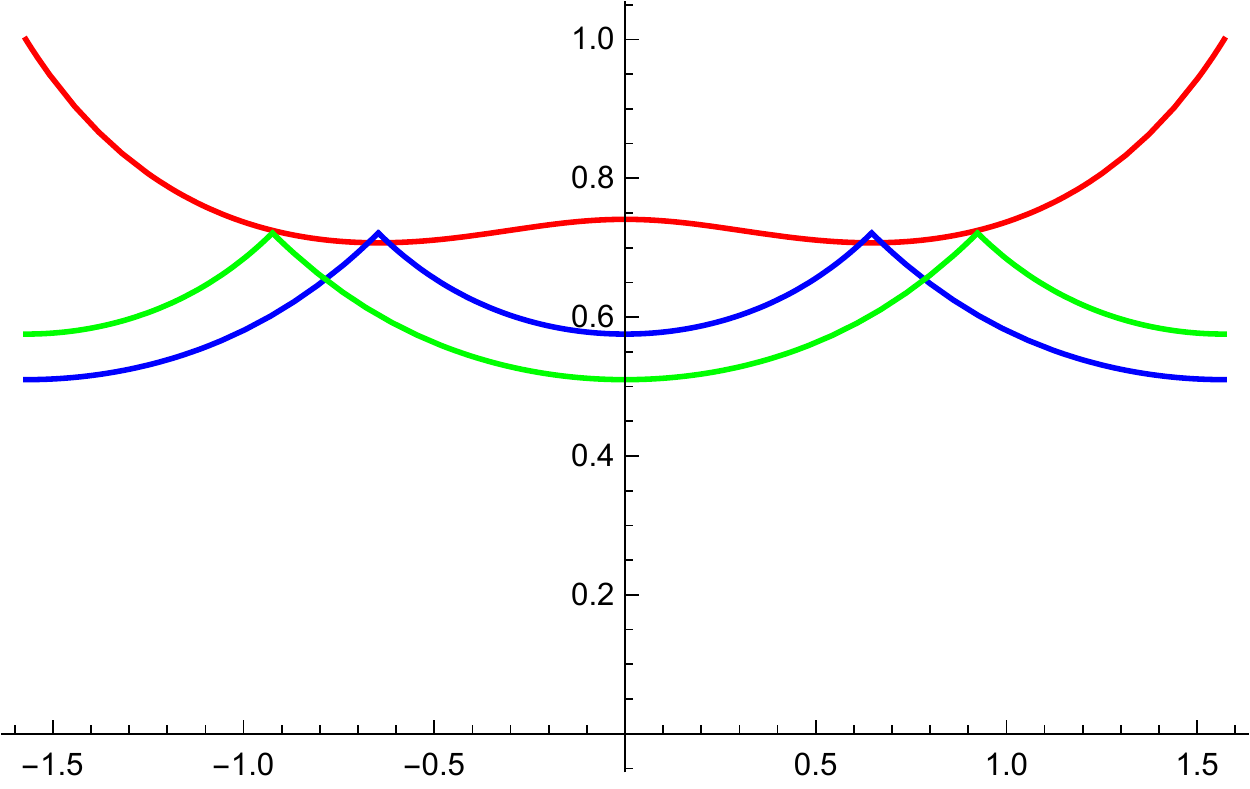}
\caption{The radial functions of the sections of the cone (red), cylinder (blue) and the rotation of the cylinder by $\pi/2$ (green). In both figures, $\te \in (\pi/4,\te_1)$. On the left, $\te$ is close to $\pi/4$; on the right, $\te$ is close to $\te_1$.  }
\label{theta1}
\end{figure}

Let 
\[
        \widetilde{\rho}_{C_\te}(u)=\left\{ \begin{array}{cc}
\frac{r \csc(u)}{\sqrt{\cot^2(u)+\cos^2(\te)}} & 0  \leq u \leq \pi/2-u_0, \\
   r\csc(u)\csc(\te)    & \pi/2-u_0 \leq u \leq \pi/2, \\
\end{array}
\right.
\]
be the radial function of the rotation by $\pi/2$ of the section of the cylinder  for $\te \in (\pi/4,\pi/2]$.
Observing Figures \ref{3secs} and \ref{theta1}, we notice that the sections of the cone and the cylinder will touch first at the ``corner'' point $u=\pi/2-u_0$, where $ \widetilde{\rho}_{C_\te}(u)$  has its maximum. Thus, we will  define  $\te_1$ as the angle such that  $\widetilde{\rho}_{C_{\te_1}}(\pi/2-u_0)=\rho_{K_{\te_1}}(\pi/2-u_0)$.  As seen above,  $\widetilde{\rho}_{C_{\te_1}}(\pi/2-u_0)=\sqrt{2}r$,  while for the cone we have
\[
     \rho_{K_\te}(\pi/2-u_0)=\frac{\sqrt{2}}{(1 + \sqrt{-1 - 2 \sec(2 \te)}) \sqrt{\sin^2(\te)-\cos^2(\te)}}.
\]
These two expressions will be equal if 
\begin{eqnarray*}
   r^{-2}= (1+  \sqrt{-1 - 2 \sec(2 \te_1)})^2 \left( \sin^2(\te_1)-\cos^2(\te_1) \right)\\
    =2 - 2 \cos(2 \te_1) \sqrt{-1 - 2 \sec(2 \te_1)},
\end{eqnarray*}
or equivalently, $-4 \cos(2 \te) \, (2 + \cos(2 \te))=(2-r^{-2})^2$, which is a quadratic equation on $ \cos(2 \te)$, with solutions $-1 \pm \sqrt{1 -\frac{(2-r^{-2})^2}{4}}$. Only the positive sign makes sense, and we obtain that the two radial functions are equal at $u=\pi/2-u_0$ only for $\te=\te_1$, where
\[
    \te_1=\frac{1}{2}\arccos \left(-1 +\frac{ \sqrt{4 r^2 - 1}}{2 r^2} \right).
\]

\bigskip
 
Now we compute $\te_2$. Let $\widehat{\rho}_{C_\te}(u)=\rho_{C_\te}(u-u_0)$. By the above considerations on $\rho_{C_\te}$, the two absolute maxima of $\widehat{\rho}_{C_\te}$  happen at $u=0$ and $u=2u_0$; the local minima happen at $u=-\pi/2+u_0$ and at $u= u_0$, (see Figure  \ref{closeto45}). At the point $u=0$ where $\widehat{\rho}_{C_\te}$ has a maximum with value $\sqrt{2}r$,  $\rho_{K_\te}$ has a local maximum with value $1/(\sin \te+ \cos \te)$. The two values coincide for $\te_2= \frac{1}{2} \arcsin \left( 1/(2r^2)-1 \right)$, and $\widehat{\rho}_{C_\te}(0)<\rho_{K_{\te}}(0)$ for $\te> \frac{1}{2} \arcsin \left( 1/(2r^2)-1 \right)$. We claim that   $\widehat{\rho}_{C_\te}(u)<\rho_{K_{\te}}(u)$ for every $u\in[-\pi/2,\pi/2]$ and $\te \in (\te_2,\pi/2]$. In fact, the slope at $u=0$ for $\rho_K$ is zero, while for $\widehat{\rho}'_{C_\te}(0+)$ is negative, so it decreases faster; both functions attain their local minimum at $u=u_0$, with $\widehat{\rho}_{C_\te}(u_0)=r \csc \te$ and $\rho_{K_\te}(u_0)=1/\sqrt{2}$. But $r \csc \te_2 <1/\sqrt{2}$ for $r\in  (1/2, \sqrt{2-\sqrt{3}})$, and $r\csc \te$ is decreasing in $\te$. Hence the cylinder function stays below the cone up to $u=u_0$. And at the other maximum for the cylinder, $\widehat{\rho}_{C_\te}(2u_0)<\rho_{K_{\te}}(2u_0)$.  

\begin{figure}[h!]
\centering
\includegraphics[scale=.5,page=1]{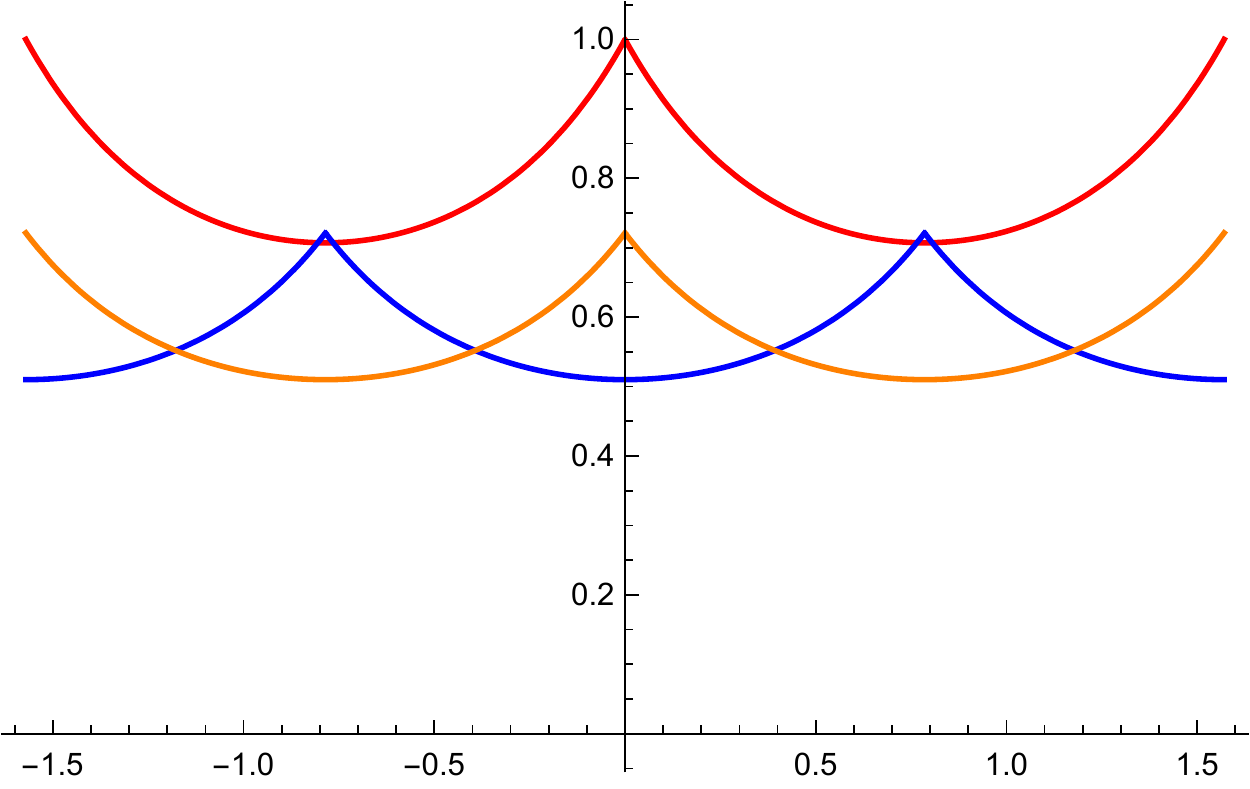}
\includegraphics[scale=.5,page=1]{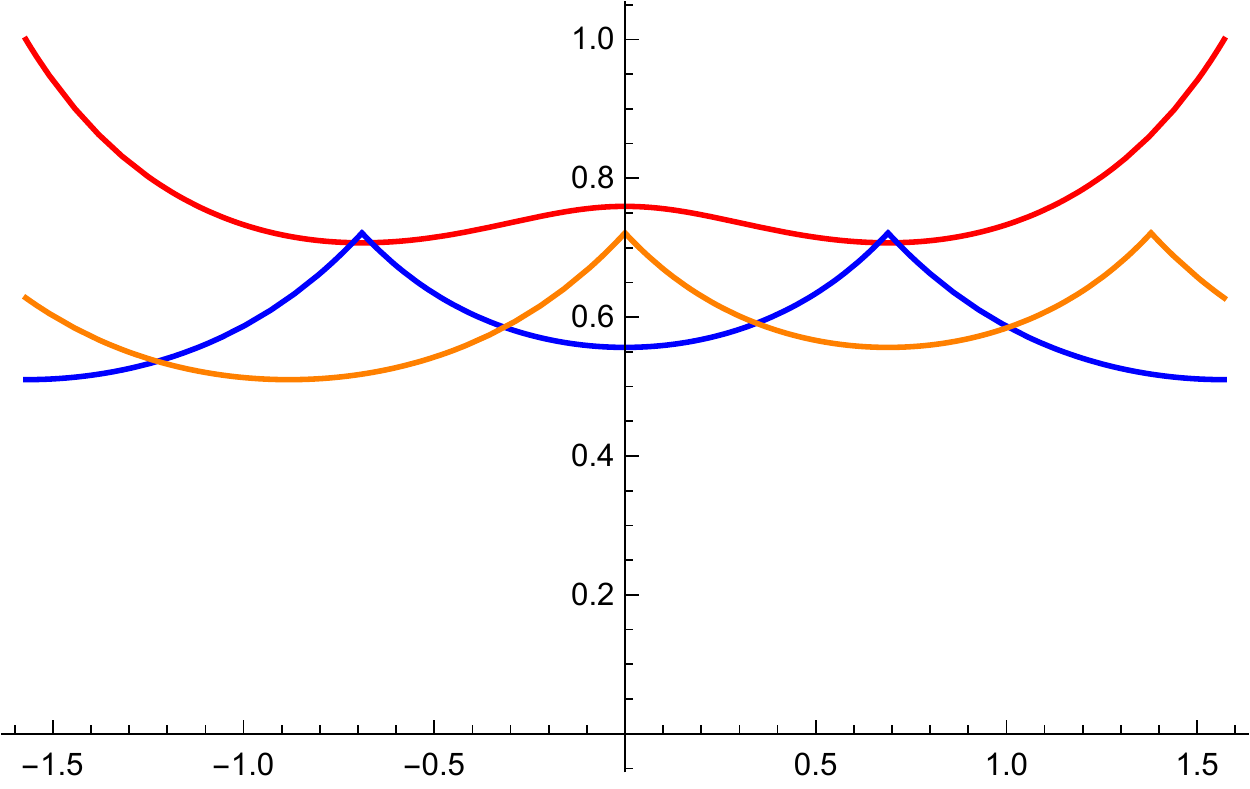}
\caption{The radial functions of the sections of the cone (red), cylinder (blue) and the rotation of the cylinder by $u_0$ (orange). The left figure shows the case $\te=\pi/2$, and the right one $\te=\te_2$.}
\label{closeto45}
\end{figure}

\bigskip

Finally, let us check that   $\te_2 <\te_1$ for $r\in (1/2, \sqrt{2-\sqrt{3}})$.  Indeed, $\cos (2\te_1)=-1+\frac{ \sqrt{4 r^2 - 1}}{2 r^2}$, while $\cos(2\te_2)= -\frac{ \sqrt{4 r^2 - 1}}{2 r^2}$, and the angles will be equal if $\frac{ \sqrt{4 r^2 - 1}}{ r^2}=1$, or $r^4-4r^2+1=0$, which has solutions $r=\pm \sqrt{2 \pm \sqrt{3}}$. Since for $r=1/2$, $\pi/4=\te_2<\te_1=\pi/2$, the same relation holds for $r \in (1/2, \sqrt{2-\sqrt{3}})$. We have proved that all sections of the cylinder can be rotated into the corresponding section of the double cone.

\bigskip

\begin{lemma}
  \label{no3Drot}
 No three-dimensional rotation of the cylinder $C$ fits inside the cone $K$.
\end{lemma}

\begin{proof}
By construction, $C \nsubseteq K$.  Since both $C$ and $K$ are origin symmetric and rotational symmetric, it is enough to consider rotations of $C$ around the $x$-axis by an angle  $\varphi \in (0,\frac{\pi}{2}]$. We will show that for each angle $\varphi \in (0,\frac{\pi}{2}]$, there is a point $P(\varphi)$ on the top rim of $C$,  that remains outside of  $K$ after a rotation by the angle $\varphi$ around the $x$-axis. Consider the point $P(\varphi)=(r\cos\alpha_0, r\sin\alpha_0,r)$,   where $\alpha_0=\arcsin \left( \frac{1-\cos\varphi}{\sin\varphi} \right)$. The rotation of angle $\varphi$ maps $P(\varphi)$ to the point  $R(\varphi)=(r\cos\alpha_0,r\sin\alpha_0\cos\varphi-r\sin\varphi,r\sin\alpha_0\sin\varphi+r\cos\varphi)
=\left( \frac{r\sqrt{\sin^2\varphi-(1-\cos\varphi)^2}}{\sin\varphi},\frac{r(\cos\varphi-1)}{\sin\varphi},r\right)$.  Note that the $z$-coordinate is positive, hence it will be enough to show that $R(\varphi)$ is outside the top part of the cone $K$, whose equation is $z=1-\sqrt{x^2+y^2}$.  But it is clear that  $1-\sqrt{\left( \frac{r\sqrt{\sin^2\varphi-(1-\cos\varphi)^2}}{\sin\varphi} \right)^2+\left( \frac{r(\cos\varphi-1)}{\sin\varphi} \right)^2} = 1-r < \frac{1}{2} < r$.  Therefore, $R(\varphi)$ is outside the cone and no three-dimensional rotation of the cylinder fits inside the cone.

\end{proof}

\end{appendix}

\end{document}